\documentclass[12pt,reqno]{amsart}
\usepackage{latexsym}
\usepackage{amssymb,amsfonts,amsmath,mathrsfs}
\newtheorem{theorem}{Theorem}[section]
\newtheorem{observation}{Remark}[section]
\newtheorem{proposition}[theorem]{Proposition}
\newtheorem{corolary}[theorem]{Corollary}
\newtheorem{lemma}[theorem]{Lemma}

\newtheorem*{theorem1.2}{Theorem 1.2}
\newtheorem*{theorem1.3}{Theorem 1.3}
\newcommand{\tmop}[1]{\ensuremath{\operatorname{#1}}}
\numberwithin{equation}{section}
\newcommand{\pr}{\mathbb{P}}
\newcommand{\R}{\mathcal{R}}
\newcommand{\s}{\mathcal{S}}

\newcommand{\mathd}{\mathrm{d}}

\newcommand{\ex}{\mathbb{E}}

\begin{document}

\title[Discrepancy bounds for the distribution of the Riemann zeta-function]
{Discrepancy bounds for the distribution of the \\ Riemann zeta-function and applications}

\author[Y. Lamzouri]{Youness Lamzouri}
\address{Department of Mathematics and Statistics, York University,
4700 Keele Street, Toronto,
 ON M3J 1P3, Canada}
\email{lamzouri@mathstat.yorku.ca}

\author[S. Lester]{Stephen Lester}
\address{School of Mathematical Sciences, Tel Aviv University, Tel Aviv 69978, Israel}
\email{slester@post.tau.ac.il}

\author[M. Radziwi\l \l]{Maksym Radziwi\l \l}
\address{School of Mathematics, Institute for Advanced Study, 
1 Einstein Drive, Princeton, NJ 08540 USA}
\email{maksym@ias.edu}

\thanks{The first author is supported in part by an NSERC Discovery grant. The
third author is partially supported by NSF grant DMS-1128155. The research leading to these results has received funding
from the European Research Council under the European Union's
Seventh Framework Programme (FP7/2007-2013) / ERC grant agreement
n$^{\text{o}}$ 320755.}

\date{}

\begin{abstract}
We investigate the distribution of the Riemann zeta-function
on the line $\tmop{Re}(s)=\sigma$.
For $\tfrac12<\sigma\le 1$ we obtain an upper bound on the discrepancy between 
the distribution of $\zeta(s)$
and that of its random model, improving results of Harman and Matsumoto.
Additionally,
we examine the distribution of the extreme values of $\zeta(s)$ inside of the critical
strip, strengthening a previous result of the first author.

As an application of these results we obtain the first effective error term
for the number of solutions to $\zeta(s) = a$ in a strip $\tfrac12 < \sigma_1 <
\sigma_2 < 1$. Previously in the strip $\tfrac 12 < \sigma < 1$ 
only an asymptotic estimate was available due
to a result of Borchsenius and Jessen from 1948 and effective estimates were known only slightly to the left of the half-line, under the Riemann hypothesis (due to Selberg) and to the right of the abscissa of absolute convergence (due to Matsumoto).  
 In general our results are an improvement of the
classical Bohr-Jessen framework and are also applicable to counting the zeros of the
Epstein zeta-function. 
\end{abstract}

\subjclass[2010]{Primary 11M06.}

\keywords{}

\maketitle

\section{Introduction and statement of main results}
Let $\{ X(p) \}_p$ be a sequence of independent random variables
uniformly distributed on the unit circle where 
 $p$ runs over the prime numbers. 
Consider the random Euler product
\[
 \zeta(\sigma, X)
= \prod_p\bigg(1-\frac{X(p)}{p^{\sigma}}\bigg)^{-1},
\] 
which converges almost surely for $\sigma>\tfrac12$.
Due to the unique factorization
of the integers we intuitively 
expect that the functions $p^{-it}$ interact like the independent random variables $X(p)$.
This suggests that $\zeta(\sigma, X)$ should
be a good model for the Riemann zeta-function, and
one may ask: How well does the distribution of $\zeta(\sigma, X)$
approximate that of the Riemann zeta-function?

A theorem of Bohr and Jessen ~\cite{BohrJ}
asserts that $\log \zeta(\sigma+it)$ has a continuous limiting distribution
in the complex plane for $\sigma>\tfrac12$. In fact,
it can be seen from their work that $\log \zeta(\sigma+it)$
converges in distribution to $\log \zeta(\sigma, X)$ for $\sigma>\tfrac12$.
In this article we investigate the discrepancy
between the distributions of the random variable $\log \zeta(\sigma, X)$
and that of $\log \zeta(\sigma+it)$, \textit{i.e.}
\[
D_{\sigma}(T):=\sup_{\mathcal R} \bigg|
\mathbb{P}_{T} \Big (\log \zeta(\sigma + it) \in \mathcal{R} \Big ) - \mathbb P\Big( \log \zeta(\sigma, X) \in \mathcal R \Big) \bigg|,
\]
where the supremum is taken over rectangles $\mathcal R$
with sides parallel to the coordinate axes and where
$$
\mathbb{P}_{T} \Big (f(t) \in \mathcal{R} \Big ) :=
\frac{1}{T} \cdot \text{meas} \big \{ T \leq t \leq 2T: f(t) \in 
\mathcal{R} \big \}. 
$$
This quantity measures the extent to which
the distribution function of the random variable $\log \zeta(\sigma, X)$ 
approximates that of $\log \zeta(\sigma+it)$.
We prove
\begin{theorem} \label{discrepancy thm}
Let $\tfrac12< \sigma<1$ be fixed. Then
\[
D_{\sigma}(T) \ll \frac{1}{(\log T)^{\sigma}}.
\]
Additionally, for $\sigma=1$ we have
\[
D_{1}(T) \ll \frac{\log \log T}{\log T}.
\]
\end{theorem}
Theorem \ref{discrepancy thm} improves upon a previous discrepancy estimate
due to G. Harman and K. Matsumoto ~\cite{HM}.  For fixed $\tfrac12 < \sigma  \le 1$ they showed that the discrepancy
satisfies the bound
\[
D_{\sigma}(T) \ll \frac{1}{(\log T)^{(4\sigma-2)/(21+8\sigma)-\varepsilon}}
\]
for any $\varepsilon>0$. One new feature of our estimate
is that the power of the logarithm does not decay to zero as
$\sigma \rightarrow \tfrac 12$. We introduce a different technique to study this problem
that relies upon careful analysis of large complex moments
of the Riemann zeta-function inside of the critical strip. 
Some of the tools developed by A. Selberg to study the distribution
 of  $\log \zeta(\tfrac12+it)$ are also used, such as Beurling-Selberg functions.

An important problem in the theory of the Riemann zeta-function
is to understand its maximal
order within the critical strip. The Riemann hypothesis implies that for $\tfrac12 <\sigma<1$ and $t$ large we have $\log|\zeta(\sigma+it)|\ll (\log t)^{2-2\sigma+o(1)}$ (see Theorem 14.5 of \cite{Titchmarsh}). On the other hand, Montgomery \cite{Montgomery} proved that $\log|\zeta(\sigma+it)|= \Omega\left((\log t)^{1-\sigma+o(1)}\right)$, and based on a probabilistic argument, he conjectured that this omega result is in fact optimal, namely that $\log|\zeta(\sigma+it)|\ll (\log t)^{1-\sigma+o(1)}.$
This motivates the study of the extent  to which the
extreme values of $\zeta(\sigma+it)$ can be modeled
by those of the random variable $\zeta(\sigma,X)$. 
For if the distribution of the extreme values of $\zeta(\sigma+it)$ matches 
that of $\zeta(\sigma,X)$ in the viable range then Montgomery's conjecture
 follows.

In \cite{LamzouriLarge}
 the first author obtained an asymptotic estimate for
$
\log \mathbb{P}_{T}(\log |\zeta(\sigma + it)| > \tau)
$
in nearly the full conjectured range of $\tau$. More precisely, he showed that there is a positive constant $A(\sigma)$, such that
uniformly in the range $\tau\ll (\log T)^{1-\sigma+o(1)}$, we have 
\begin{equation}\label{large deviation asymptotic}
\begin{aligned}
\log\mathbb{P}_{T} \Big (\log |\zeta(\sigma + it)| > \tau \Big )
&=(1+o(1))\log\mathbb{P} \Big ( \log |\zeta(\sigma, X)| > \tau \Big)\\
&=\left(-A(\sigma)+o(1)\right)\tau^{\frac{1}{1-\sigma}}(\log \tau)^{\frac{\sigma}{1-\sigma}}.
\end{aligned}
\end{equation}
 We strengthen this result, obtaining an asymptotic formula for
$\mathbb{P}_{T}(\log |\zeta(\sigma + it)| > \tau)$ in the same
range. 
\begin{theorem}  \label{large thm}
Let $\tfrac12 < \sigma<1$ be fixed.
There exists a constant $b(\sigma)>0$
such that for $3 \leq \tau \leq b(\sigma) (\log T)^{1-\sigma}(\log
\log T)^{1-\frac{1}{\sigma}}$ we have
\begin{equation*}
\mathbb{P}_{T} \Big (\log |\zeta(\sigma + it)| > \tau \Big )
= \mathbb{P} \Big ( \log |\zeta(\sigma, X)| > \tau \Big )
\times \bigg (1 + O\bigg ( \frac{(\tau \log \tau)^{\frac{\sigma}{1 - \sigma}}
\cdot (\log\log T)}{(\log T)^{\sigma}} \bigg ) \bigg ).
\end{equation*}
Moreover, the same asymptotic estimate holds when $\log|\zeta(\sigma+it)|$ and $\log|\zeta(\sigma,X)|$ are replaced by $\arg\zeta(\sigma+it)$ and $\arg\zeta(\sigma, X)$ respectively.
\end{theorem}

The terms $(\log T)^{\sigma}$ appearing in the error term in Theorem \ref{large thm} and
in Theorem \ref{discrepancy thm} are related. An improvement in our method would produce an improvement in both
results. Since we do not believe that we will be able to extend significantly the range of 
Theorem \ref{large thm}, it seems that our bound for $D_{\sigma}(T)$ is as well optimal given the method used. 

We also apply Theorem \ref{discrepancy thm} to study
the roots, $s$,
to the equation $\zeta(s)=a$ where $a$ is a nonzero complex number.
These points are known as $a$-points and the study of their distribution
is a classical topic in the theory of the Riemann zeta-function.

Let $N_a(\sigma_1,\sigma_2; T)$ be the number of $a$-points in the strip 
$\tfrac12 < \sigma_1<\sigma<\sigma_2<1$, $T\le t\le 2T$.
In 1948 Borchsenius and Jessen
 \cite{BorchseniusJessen}
proved that there exists a constant $c(a, \sigma_1,\sigma_2)>0$ such that
as $T \rightarrow \infty$
\begin{equation}\label{BorchseniusJessen}
N_a(\sigma_1,\sigma_2; T)\sim c(a, \sigma_1,\sigma_2)T.
\end{equation}
The constant $c(a, \sigma_1,\sigma_2)$ can be explicitly given in terms of the random variable
 $\zeta(\sigma,X)$. Indeed, let
\[
f_a(\sigma)=\mathbb E\Big( \log| \zeta(\sigma, X)-a| \Big).
\]
Then,
\[
c(a,\sigma_1,\sigma_2)=\frac{f_a'(\sigma_2)-f_a'(\sigma_1)}{2\pi}.
\]
The differentiability of $f_a(\sigma)$ is not trivial, and was established
by Borchsenius and Jessen. 

Using Theorem \ref{discrepancy thm} we obtain the first effective error term for
$N_a(\sigma_1,\sigma_2; T)$ valid for $\sigma_1 < \sigma_2$ in the
critical strip.
\begin{theorem} \label{apoint thm}
Let $\tfrac12<\sigma_1<\sigma_2<1$.
For every nonzero complex number $a$
there exists a constant $c(a, \sigma_1,\sigma_2)>0$
such that
\[
N_a(\sigma_1,\sigma_2; T)=c(a,\sigma_1,\sigma_2) T+O\bigg(T \cdot \frac{\log \log T}{(\log T)^{\sigma_1/2}} \bigg).
\]
\end{theorem}
Inside the critical strip,
an effective error term was known previously only slightly to the left of the half-line under the 
assumption of the Riemann hypothesis, thanks to 
unpublished work of Selberg (see \cite{Tsang}, Chapter 8).  
In the region of absolute convergence ($\sigma > 1$),
Matsumoto \cite{Matsumotoapoints2}, \cite{Matsumotoapoints3}, with some additional constraints,
has given a formula for the number
of $a$-points of $\log \zeta(s)$ with an error term that has a power saving of $\log \log T$.
We have not determined the limits of our method for $\sigma>1$, 
but it should give a formula 
for the number of $a$-points of $\zeta(s)$ (and $\log \zeta(s)$ as well)
with an error term with a saving of at least $(\log T)^{1/2}$.

The error term in Theorem \ref{apoint thm}
is essentially the
square-root of the discrepancy $D_{\sigma}(T)$. 
We have not been able to 
determine conjecturally the correct size of $D_{\sigma}(T)$. In this direction
we have only the following remark. 
\begin{observation}
We have,
$$
D_{\sigma}(T) = \Omega(T^{1 - 2\sigma - \varepsilon}).
$$
Moreover, If $D_{\sigma}(T) = O(T^{1 - 2\sigma + \varepsilon})$ then
the Zero Density Hypothesis holds. 
\end{observation}
We give a proof of this remark in the Appendix. There is an apparent discrepancy between our lower and upper bound for $D_{\sigma}(T)$. It would be
very interesting to work out a reliable heuristic to predict the correct size of $D_{\sigma}(T)$. 

It is likely that our ideas can be generalized to other situations where the Bohr-Jessen framework
applies \cite{BohrJ}. For example, our method should adapt to the study of the zeros of the
Epstein zeta-function of a quadratic form with class number $n$. We expect the method to show
that the number of zeros of the Epstein zeta-function in the strip $\sigma_1 < \sigma < \sigma_2$
is $c(\sigma_1,\sigma_2) T + O(T (\log T)^{-A(\sigma_1,n)})$ with an $A(\sigma_1,n) \rightarrow 0$
as $n \rightarrow \infty$. This would refine previous results of Bombieri and Mueller \cite{BombieriMueller}
and Lee's improvement there-of \cite{Lee}.

\section{Key ideas and detailed results}

In probability theory, the classical Berry-Esseen Theorem states that if the characteristic functions of two real valued random variables are close, then their corresponding probability distributions are close as well.  In order to establish Theorem \ref{discrepancy thm} the key ingredient is to show that the characteristic function of the joint distribution of $\tmop{Re}\log\zeta(\sigma+it)$ and $\tmop{Im}\log\zeta(\sigma+it)$ can be very well approximated by the corresponding characteristic function of the random model $\log\zeta(\sigma,X)$. For $u,v\in \mathbb{R}$ we define
$$ 
\Phi_{\sigma, T}(u,v):= \frac{1}{T} \int_T^{2T} 
\exp\Big(i u \tmop{Re}\log\zeta(\sigma+it)+iv \tmop{Im}\log\zeta(\sigma+it)\Big)dt,
$$
and
$$ \Phi_{\sigma}^{\textup{rand}}(u,v):= \mathbb E \bigg(\exp\Big( i u \tmop{Re} \log \zeta(\sigma,X)
+iv \tmop{Im} \log \zeta (\sigma,X) \Big) \bigg).$$
Then we prove
\begin{theorem} \label{characteristic thm}
Let $\tfrac12< \sigma <1$ and $A\geq1$ be fixed. 
There exists a constant $b_1=b_1(\sigma,A)$
such that for $|u|,|v| \leq  b_1(\log T)^{\sigma}$ we have
\begin{equation}\label{asympChar}
 \Phi_{\sigma, T}(u,v)=\Phi_{\sigma}^{\textup{rand}}(u,v)+O\left(\frac{1}{(\log T)^A}\right).
\end{equation}
Moreover, there exists a constant $b_2=b_2(A)$ such that for $|u|,|v| \leq b_2\log T/\log\log T$ we have
$$ \Phi_{1, T}(u,v)=\Phi_{1}^{\textup{rand}}(u,v)+O\left(\frac{1}{(\log T)^A}\right).
$$
\end{theorem}
To deduce Theorem \ref{discrepancy thm} we use Beurling-Selberg functions (see Section 6 below) to relate the distribution function $\mathbb{P}_{T} (\log \zeta(\sigma + it) \in \mathcal{R})$ to the characteristic function $\Phi_{\sigma, T}(u,v)$. We should note that any improvement in the range of validity of Theorem \ref{characteristic thm} would lead to an improved bound for the discrepancy $D_{\sigma}(T)$. Indeed, we can show that $D_{\sigma}(T)\ll 1/L$ if the asymptotic formula \eqref{asympChar} holds in the range $|u|,|v| \ll L$.

In order to investigate the distribution of large values of $\log|\zeta(\sigma+it)|$ (or $\arg\zeta(\sigma+it)$) and prove Theorem \ref{large thm}, we study large complex moments of $\zeta(\sigma+it)$ and compare them to the corresponding complex moments of $\zeta(\sigma,X)$. Define
$$M_z(T):=\frac{1}{T}\int_T^{2T}|\zeta(\sigma+it)|^zdt.$$
Assuming the Riemann hypothesis, the first author \cite{LamzouriLarge} established an asymptotic formula for $M_z(T)$ uniformly in the range $|z|\ll (\log T)^{2\sigma-1}$, and conjectured that such an asymptotic should hold in the extended range $|z|\ll (\log T)^{\sigma}$. The assumption of the Riemann hypothesis is necessary in this case, since $|\zeta(\sigma+it)|^z$ is very large when $\sigma+it$ is close to a zero of $\zeta(s)$ and $z$ is a negative real number. Also note that, when $\tmop{Re}(z)$ is large, the moment $M_z(T)$ is heavily affected by the contribution of the points $t$ where $|\zeta(\sigma+it)|$ is large. Thus, short of proving strong bounds for $|\zeta(\sigma+it)|$ and without assuming the Riemann hypothesis, we cannot hope for asymptotics of the moments $M_z(T)$, except in a narrow range of values for $z$. To overcome this difficulty, we compute instead complex moments of $\zeta(\sigma+it)$ after first removing a small set of ``bad'' points $t$ in $[T,2T]$, namely those close to zeros of $\zeta(s)$ and those for which $|\zeta(\sigma+it)|$ is large. Using this method we obtain an asymptotic formula for these moments in the full conjectured range $|z|\ll (\log T)^{\sigma}$.

\begin{theorem}\label{complex theorem}
Let $\tfrac 12 < \sigma < 1$ and $A\geq 1$ be fixed. There exist positive constants $b_3=b_3(\sigma, A)$ and $b_4=b_4(\sigma, A)$ and a set  $\mathcal{E}(T)\subset [T,2T]$ of measure $\textup{meas}({\mathcal{E}(T)})\leq T\exp\big(-b_3\log T/\log\log T \big),$
such that for all complex
numbers $z$ with $|z| \leq b_4(\log T)^{\sigma}$ 
we have
$$
\frac{1}{T} \int_{[T,2T]\setminus\mathcal{E}(T)} |\zeta(\sigma+it)|^zdt  
= \ex\left(|\zeta(\sigma,X)|^z\right) +O\left(\frac{\ex\left(|\zeta(\sigma,X)|^{\tmop{Re}(z)}\right)}{(\log T)^A}\right),
$$
Moreover, the same asymptotic formula holds when $|\zeta(\sigma+it)|^z$ and $|\zeta(\sigma,X)|^z$ are replaced by $\exp\big(z(\arg\zeta(\sigma+it))\big)$ and $\exp\big(z(\arg\zeta(\sigma,X))\big)$ respectively.
\end{theorem}

When computing complex moments of $\zeta(\sigma+it)$ the first step is to use the classical zero density estimates to approximate $\log\zeta(\sigma+it)$ by a short Dirichlet polynomial for all $t\in [T,2T]$ except for a set of small measure (see Lemma \ref{gs formula} below). Let
$$ R_Y(\sigma+it):= \sum_{p^n\leq Y} \frac{1}{np^{n(\sigma+it)}} \quad \textup{ and }  \quad R_Y(\sigma, X):= \sum_{p^n\leq Y} \frac{X(p)^n}{np^{\sigma n}}.$$
We extract Theorems \ref{characteristic thm} and \ref{complex theorem} from the following key proposition.
\begin{proposition}\label{complex}
Let $\tfrac 12 < \sigma < 1$ and $A\geq 1$ be fixed. Let $Y=(\log T)^A$. There exist positive constants $b_5=b_5(\sigma, A)>0$ and $b_6=b_6(\sigma, A)$ such that for all complex
numbers $z_1,z_2$ with $|z_1|, |z_2| \leq b_5(\log T)^{\sigma}$ 
we have
\begin{align*}
&\frac{1}{T} \int_{\mathcal{A}(T)}  \exp\Big(z_1R_Y(\sigma+it)+ z_2\overline{R_Y(\sigma+it)}\Big)dt \\
&\qquad \qquad = \ex\left(\exp\left(z_1R_Y(\sigma, X)+ z_2 \overline{R_Y(\sigma, X)}\right)\right) +O\left(\exp\left(-b_6\frac{\log T}{\log\log T}\right)\right),
\end{align*}
where $\mathcal{A}(T)$ is the set of those $t\in [T, 2T]$ such that 
$ |R_Y(\sigma+it)|\leq (\log T)^{1-\sigma}/\log \log T.$
\end{proposition}


Compared to earlier treatments our main innovation consists in
the introduction of the
condition $|R_Y(\sigma+it)| \leq (\log T)^{1 - \sigma} / \log\log T$
in $\mathcal{A}(T)$. Without this 
constraint the range of $|z_1|$ and $|z_2|$ in Proposition \ref{complex}
would be reduced to $(\log T)^{2 \sigma - 1}$. 

Using Littlewood's Lemma (see equation \eqref{Littlewood} below), one can count the number of $a$-points of $\zeta(s)$ in the strip $\tfrac12 < \sigma_1<\sigma<\sigma_2<1$, $T\le t\le 2T$, if one can estimate the integral
\begin{equation}\label{intlogzeta}
\int_{T}^{2T} \log |\zeta(\sigma + it) - a| dt.
\end{equation}
In \cite{BorchseniusJessen}, Borchsenius and Jessen proved the following asymptotic formula  for this integral from which they deduced  their result \eqref{BorchseniusJessen}
$$ 
\frac1T\int_{T}^{2T} \log |\zeta(\sigma + it) - a| dt \sim \mathbb{E} [ \log |\zeta(\sigma,X) - a| ], \textup{ as } T\to\infty.
$$
We improve on this result, obtaining the first effective error term for the integral \eqref{intlogzeta}.
\begin{theorem} \label{aux}
Let $\tfrac 12 < \sigma < 1$ and $a\neq 0$ be a complex number. Then,
$$
\frac1T\int_{T}^{2T} \log |\zeta(\sigma + it) - a| dt = \mathbb{E} [\log |\zeta(\sigma,X) - a| ]
+ O \bigg ( \frac{(\log\log T)^{2}}{(\log T)^{\sigma}}\bigg ). 
$$
\end{theorem}
We should note that apart from the factor $(\log\log T)^2$, the error term in Theorem \ref{aux} is optimal in view of our bound for the discrepancy $D_{\sigma}(T)$ in Theorem \ref{discrepancy thm}.

There are two main ingredients in the proof of Theorem \ref{aux}. 
First, we use our result on $D_{\sigma}(T)$ to capture the main term.
Secondly, to control the error term we need a completely uniform (but not necessarily very good)
bound for the measure of those $t$ for which $\zeta(\sigma + it)$ is
very close to $a$. We achieve such an estimate by using the following $L^{2k}$ bound.
\begin{proposition} \label{apoint prop}
Let $\tfrac 12 < \sigma \leq 2$ be fixed. Let $a \in \mathbb{C}$. There
exists an absolute constant $C > 0$ such that for every real number $k > 0$
we have
$$
\frac1T\int_{T}^{2T} |\log |\zeta(\sigma + it) - a||^{2k} dt \ll  (C k)^{4k}.
$$
\end{proposition}
In order to study $a$-points to the left of the half-line, Selberg obtains a similar proposition when $\sigma = \tfrac 12$. His argument
depends on the rapid rate of change of the phase of $\zeta(\sigma + it)$ when $\sigma \leq \tfrac 12$
(for $\sigma<\tfrac12$ this follows from the Riemann hypothesis) and does not generalize
to any line with $\sigma > \tfrac 12$ (see \cite{Tsang}, Chapter 8, in particular the discussion on page 119). Our treatment depends on a careful use of Jensen's formula.
Proposition \ref{apoint prop} bears some resemblance to a result obtained by
Guo to study zeros of $\zeta'(s)$. Our result is more refined, in particular our treatment removes the loss of a power of $\log\log T$.

\section{Preliminary Lemmas}

In this section we collect together several preliminary results that will be useful in 
our subsequent work. 

\begin{lemma}[Lemma 2.2 of \cite{GranvilleSound}]\label{gs formula}
Let $\tfrac12 <\sigma \leq 1$ be fixed and 
$3\leq Y\leq T/2$. 
For $t \in [T, 2T]$ except outside a 
set of measure $\ll T^{5/4-\sigma/2} Y \log^5 T$
we have
\begin{equation} \label{dirichlet poly}
\log \zeta(\sigma+it)
=R_Y(\sigma+it)
+O\big(Y^{-(\sigma-1/2)/2} \log^3 T\big).
\end{equation}
\end{lemma}

\begin{lemma} \label{diagonal}
Let $ 2 \leq y \leq z$.
For any positive integer $k$ that is
$\leq \log T/(3 \log z)$ we have
\[
\frac1T \int_T^{2T} \bigg|\sum_{y \leq p \leq z}
\frac{1}{p^{\sigma+it}} \bigg|^{2k} dt
\ll k! \bigg( \sum_{y \leq p \leq z}\frac{1}{p^{2\sigma}}\bigg)^k+T^{-1/3}.
\]
Additionally, for any positive integer $k$
we have
\[
\mathbb E\Bigg( \bigg|\sum_{y \leq p \leq z} \frac{X(p)}{p^{\sigma}} \bigg|^{2k} \Bigg)
\ll k! \bigg( \sum_{y \leq p \leq z}\frac{1}{p^{2\sigma}}\bigg)^k.
\]
\end{lemma}
\begin{proof}
The first assertion of the lemma is
Lemma 4.2 of \cite{LamzouriLarge}.  
Next, note that
\[
\mathbb E\Bigg( \bigg|\sum_{y \leq p \leq z} \frac{X(p)}{p^{\sigma}} \bigg|^{2k} \Bigg)
=\sum_{\substack{y \leq p_1, \ldots, p_k \leq z \\
y \leq q_1, \ldots, q_k \leq z}}
\frac{1}{(p_1 \cdots p_k q_1 \cdots q_k)^{\sigma}}
\mathbb E\Big(X(p_1)\cdots X(p_k) \overline{X(q_1)
\cdots X(q_k)}\Big).
\]
Since the $X(p)$'s are independent random variables
uniformly distributed
on the unit circle the
only terms that contribute to the above sum are
those where $p_1\cdots p_k=q_1\cdots q_k$. The
contribution from these terms is
\[
\ll
k! \bigg( \sum_{ y \leq p \leq z} \frac{1}{p^{2\sigma}}\bigg)^{k}.
\]

\end{proof}

\begin{lemma} \label{large sieve} Let $\tfrac12 < \sigma < 1$ and $A\geq 1$ be fixed. Also, let $Y=(\log T)^A$ and  $k $ be an integer that satisfies $2 \le k \leq  \log T/(6 A \log \log T)$.
Then there exists a constant $a(\sigma)>0$ such that
\[
\frac1T \int_T^{2T} \left|R_Y(\sigma+it)\right|^{2k} dt 
\ll \bigg(\frac{a(\sigma) k^{1-\sigma}}{ (\log k )^{\sigma}}\bigg)^{2k}.
\]
Additionally, for any integer $k \ge 2$ we have
\[
\mathbb E\left(\left|R_Y(\sigma,X)\right|^{2k} \right)
\ll \bigg(\frac{a(\sigma) k^{1-\sigma}}{(\log k)^{\sigma}} \bigg)^{2k}.
\]
\end{lemma}
\begin{proof}
We will only prove the first assertion;
the second follows from a similar argument.
Plainly,
\begin{equation} \label{diophantine bd}
\begin{split}
\int_T^{2T} \left|R_Y(\sigma+it) \right|^{2k} dt  \leq & 9^k \bigg(\int_T^{2T} \bigg|
\sum_{p \leq k \log k} \frac{1}{p^{\sigma+it}} \bigg|^{2k} dt
+
\int_T^{2T} \bigg|
\sum_{k \log k \leq p \leq Y} \frac{1}{p^{\sigma+it}} \bigg|^{2k} dt\\
& \qquad
+O(T \log^{2k} \zeta(2\sigma))\bigg).
\end{split}
\end{equation}
By Lemma \ref{diagonal} and the prime number theorem
we have
\[
\frac1T \int_T^{2T} \bigg|\sum_{k \log k \leq p \leq Y} 
\frac{1}{p^{\sigma+it}} \bigg|^{2k} dt \ll k! \bigg(\sum_{k \log k
\leq p \leq Y} \frac{1}{p^{2\sigma}} \bigg)^{k}+T^{-1/3} \ll 
k^k \bigg(
\frac{(k \log 2k)^{1-2\sigma}}{(2\sigma-1) \log k}\bigg)^{k}.
\]
Next, note that for fixed $1/2 < \sigma<1$
\[
\frac1T \int_T^{2T} \bigg|\sum_{p \leq k \log k} 
\frac{1}{p^{\sigma+it}} \bigg|^{2k} dt \leq 
 \bigg(\sum_{p \leq k \log 2k} \frac{1}{p^{\sigma}} \bigg)^{2k}
\ll  \bigg(\frac{(k \log k)^{1-\sigma}}{(1-\sigma)\log k}
\bigg)^{2k},
\]
by the prime number theorem. 
Inserting the two estimates above into \eqref{diophantine bd} completes the proof.
\end{proof}

\begin{lemma}\label{poly large deviation}
Let $\tfrac12 <\sigma<1$ and $A\geq 1$ be fixed, and  let $Y=(\log T)^A$. Then there exists a constant $B=B(\sigma,A)$ such that  
$$ \pr_T\left(\left|R_Y(\sigma+it)\right|\geq \frac{(\log T)^{1-\sigma}}{\log\log T}\right)\ll \exp\left(-B\frac{\log T}{\log \log T}\right) $$
and $$ \pr\left(\left|R_Y(\sigma,X)\right|\geq \frac{(\log T)^{1-\sigma}}{\log\log T}\right)\ll \exp\left(-B\frac{\log T}{\log \log T}\right).$$
\end{lemma}
\begin{proof}
We will only prove the first assertion; the second follows from a similar argument. 
Let $2 \le k\leq \log T/(6A\log\log T)$ be an integer. Then, Lemma \ref{large sieve} implies that 
\begin{align*}
 \pr_T\left(\left|R_Y(\sigma+it)\right|\geq \frac{(\log T)^{1-\sigma}}{\log\log T}\right) &\leq \left(\frac{(\log T)^{1-\sigma}}{\log\log T}\right)^{-2k} \frac1T \int_T^{2T} \left|R_Y(\sigma+it)\right|^{2k} dt\\
& \ll  \left( \frac{C k^{1 - \sigma} \log\log T}{(\log k)^{\sigma} (\log T)^{1-\sigma}} \right)^{2k}.
\end{align*}
Choosing $k= [\log T/(C_1\log\log T)]$, where $C_1=6A (1+C)^{1/(1-\sigma)}$, yields the desired bound.
\end{proof}

\begin{lemma} \label{transition lemma}
Let $\tfrac12 <\sigma\leq 1$ and $A\geq 1$ be fixed, and  let $Y=(\log T)^A$. Then, for any positive integers $k$, $\ell$ such that 
$k+\ell\leq (\log T)/(6 A \log\log T)$, we have
\begin{align*}
& \frac{1}{T}
\int_{T}^{2T} \bigg (R_Y(\sigma+it)\bigg)^{k}
\cdot \bigg (\overline{R_Y(\sigma+it)} \bigg )^{\ell} dt
\\ 
& \qquad \qquad \qquad \qquad \qquad 
 = \mathbb{E} \bigg( \bigg(R_Y(\sigma,X)\bigg)^{k}
\cdot \bigg (\overline{R_Y(\sigma,X)}\bigg )^{\ell}
\bigg ) + O \bigg ( \frac{Y^{k + \ell}}{\sqrt{T}} \bigg ). 
\end{align*}
\end{lemma}
\begin{proof}
See Lemma 3.4 in Tsang's thesis \cite{Tsang}.
\end{proof}

\section{Complex moments of $\zeta(\sigma+it)$: Proofs of
Theorems \ref{characteristic thm} and \ref{complex theorem}}

We begin by proving Proposition  \ref{complex}.
\begin{proof}[Proof of Proposition \ref{complex}]
Let $k=\max\{|z_1|, |z_2|\}$, and $N=[\log T/(D(\log\log T))]$ where $D$ is a suitably large constant. 
Then, we have 
\begin{equation}\label{taylor}
\begin{aligned}
&\frac{1}{T} \int_{\mathcal{A}(T)} \exp\Big(z_1 R_Y(\sigma+it)+z_2\overline{R_Y(\sigma+it)}\Big) dt\\
&= \sum_{j+\ell\leq N}\frac{z_1^jz_2^{\ell}}{j!\ell!}  \frac{1}{T} \int_{\mathcal A(T)} \bigg(R_Y(\sigma+it)\bigg)^j\bigg(\overline{R_Y(\sigma+it)}\bigg)^{\ell} dt +E_1
\end{aligned}
\end{equation}
where 
\begin{align*}E_1&\ll \sum_{j+\ell\geq N} \frac{k^{j+\ell}}{j!\ell!}\frac{1}{T} 
\int_{\mathcal{A}(T)} \big|R_Y(\sigma+it)\big|^{j+\ell}\mathd t\leq \sum_{n\geq N} \frac{k^n}{n!} \left(\frac{(\log T)^{1-\sigma}}{\log\log T}\right)^n\sum_{j\leq n}\frac{n!}{j!(n-j)!}\\
&\leq \sum_{n\geq N}\frac{1}{n!}\left(\frac{2k(\log T)^{1-\sigma}}{\log\log T}\right)^{n}\leq \sum_{n\geq N}\left(\frac{6k(\log T)^{1-\sigma}}{N\log\log T}\right)^{n}\ll e^{-N},
\end{align*}
using Stirling's formula along with the fact that  
$\big|R_Y(\sigma+it)\big|\leq (\log T)^{1-\sigma}/\log\log T$ 
for $t\in \mathcal{A}(T)$.

Let $\mathcal S(T)=\{T \le t \le 2T: t \notin \mathcal A(T)\}$. 
If $j+\ell\leq N$ then using Lemmas  \ref{large sieve} and \ref{poly large deviation} along with the Cauchy-Schwarz inequality we  get
\begin{align*}
&\frac{1}{T} \int_{\mathcal{S}(T)} \bigg(R_Y(\sigma+it)\bigg)^j\bigg(\overline{R_Y(\sigma+it)}\bigg)^{\ell} dt\\
&\leq \left(\frac{\text{meas}(\mathcal{S}(T))}{T}\right)^{1/2}\left(\frac{1}{T} \int_{T}^{2T} \big|R_Y(\sigma+it)\big|^{2(j+\ell)} dt\right)^{1/2} \\
&\ll \exp\left(-B\frac{\log T}{2\log\log T}\right) 
\left(C\frac{(j+\ell)^{1-\sigma}}{(\log (j+\ell+2))^{\sigma}}\right)^{j+\ell},\\
\end{align*}
for some positive constants $B=B(\sigma, A)$ and $C=C(\sigma)$.
Inserting this bound in equation \eqref{taylor} we deduce
\begin{equation}\label{poly bd 1}
\begin{aligned}
&\frac{1}{T} \int_{\mathcal{A}(T)} \exp\Big(z_1R_Y(\sigma+it)+z_2\overline{R_Y(\sigma+it)}\Big) dt\\
&= \sum_{j+\ell\leq N}\frac{z_1^jz_2^{\ell}}{j!\ell!}  \frac{1}{T} \int_T^{2T}  \bigg(R_Y(\sigma+it)\bigg)^j\bigg(\overline{R_Y(\sigma+it)}\bigg)^{\ell}dt +E_2,
\end{aligned}
\end{equation}
where 
\begin{equation}\label{poly bd 2}
\begin{aligned}
E_2 
&\ll \exp\left(-B\frac{\log T}{2\log\log T}\right) \sum_{j+\ell\leq N}\frac{k^{j+\ell}}{j!\ell!}\left(C\frac{(j+\ell)^{1-\sigma}}{(\log (j+\ell+2))^{\sigma}}\right)^{j+\ell} +e^{-N}\\
&\ll \exp\left(-B\frac{\log T}{2\log\log T}\right) \sum_{n\leq N} \frac{1}{n!}\left(2C\frac{k N^{1-\sigma}}{(\log (N+2))^{\sigma}}\right)^{n}+e^{-N}\\
&\ll \exp\left(-B\frac{\log T}{2\log\log T}\right) \exp\left(2C\frac{k N^{1-\sigma}}{(\log (N+2))^{\sigma}} \right)+e^{-N}\\
& \ll \exp\left(-B\frac{\log T}{4\log\log T}\right) +e^{-N},
\end{aligned}
\end{equation}
if $D$ is suitably large and $k\leq c_0 (\log T)^{\sigma}$ where $c_0$ is suitably small. 

Now,  for all $j+\ell\leq N$, we have by Lemma \ref{transition lemma} that 
\begin{align*}
\frac{1}{T} \int_{T}^{2T}  \bigg(R_Y(\sigma+it)\bigg)^j\bigg(\overline{R_Y(\sigma+it)}\bigg)^{\ell} dt =& \ex\bigg( \bigg(R_Y(\sigma,X)\bigg)^j\bigg(\overline{R_Y(\sigma,X)}\bigg)^{\ell}\bigg)\\
 &+O\left(\frac{Y^{j+\ell}}{\sqrt{T}}\right).\\
\end{align*}
Note that $Y^{2N}\ll T^{1/4}$ if $D$  is suitably large. By this,
\eqref{poly bd 1}, and \eqref{poly bd 2} we obtain
\begin{equation}\label{approx random 1}
\begin{aligned}
&\frac{1}{T} \int_{\mathcal{A}(T)} \exp\Big(z_1 R_Y(\sigma+it)+z_2\overline{R_Y(\sigma+it)}\Big) dt\\
&= \sum_{j+\ell\leq N}\frac{z_1^jz_2^{\ell}}{j!\ell!}\ex\bigg( \bigg(R_Y(\sigma,X)\bigg)^j\bigg(\overline{R_Y(\sigma,X)}\bigg)^{\ell}\bigg) + O\left(\exp\left(-\frac B4\frac{\log T}{\log\log T}\right)\right).
\end{aligned}
\end{equation}
Furthermore, by Lemma \ref{large sieve}, 
$$  \ex\left(\left|R_Y(\sigma,X)\right|^k\right)\ll \left(C\frac{k^{1-\sigma}}{(\log k)^{\sigma}}\right)^k,$$
for all $k\geq 2$. 
Therefore, the main term on the right-hand side of \eqref{approx random 1} equals 
$$
\ex\left(\exp\Big(z_1 R_Y(\sigma,X)+z_2\overline{R_Y(\sigma,X)}\Big)\right)+ E_3,
$$
where 
$$E_3\ll \sum_{n\geq N} \frac{1}{n!}\left(\frac{2Ck n^{1-\sigma}}{(\log n)^{\sigma}}\right)^n\leq \sum_{n\geq N} \left(\frac{6Ck}{(n\log n)^{\sigma}}\right)^n\leq \sum_{n\geq N} \left(\frac{6Ck}{(N\log N)^{\sigma}}\right)^n\ll e^{-N}.
$$
This completes the proof.

\end{proof}
Before proving Theorems \ref{characteristic thm} and \ref{complex theorem}, we need the following lemma which shows that the characteristic function of the random variable $\log\zeta(\sigma,X)$ is well approximated by that of $R_Y(\sigma,X)$ in a certain range that depends on $Y$.

\begin{lemma}\label{TruncateRandom}
Let $Y$ be a large positive real number, and $s$ be a complex number such that $|s|\leq Y^{\sigma-1/2}$.  Then we have 
\begin{equation}\label{trunc1} \ex\left(|\zeta(\sigma, X)|^s\right) = \ex\bigg(\exp\bigg(s\tmop{Re} \big(R_Y(\sigma,X)\big)\bigg)\bigg)+ O\left(\ex\left(|\zeta(\sigma, X)|^{\tmop{Re}(s)}\right)\frac{|s|}{Y^{\sigma-1/2}}\right).
\end{equation}
Moreover, if
$u,v$ are real numbers such that $|u|+|v|\leq Y^{\sigma-1/2}$, then 
\begin{equation}\label{trunc2}
\Phi_{\sigma}^{\textup{rand}}(u,v)
= \ex\bigg(\exp\bigg(iu \tmop{Re} R_Y(\sigma,X)+ iv \tmop{Im}R_Y(\sigma,X)\bigg)\bigg) +O\left(\frac{|u|+|v|}{Y^{\sigma-1/2}}\right).
\end{equation}

\end{lemma}

\begin{proof} 
Let $z$ be a complex number with $|z|\leq Y^{\sigma-1/2}.$ Using that 
$$
\sum_{\substack{n\geq 2\\ p^n>Y}}\frac{1}{p^{\sigma n}}\ll \frac{1}{Y^{\sigma-1/2}},
$$
 we obtain
$$ \ex\bigg(\exp\Big(z\tmop{Re} \log\zeta(\sigma, X)\Big)\bigg) = \ex\left(\exp\left(z\tmop{Re} \big(R_Y(\sigma,X)\big)+z \tmop{Re}\sum_{p>Y}\frac{X(p)}{p^{\sigma}}+ O\left(\frac{|z|}{Y^{\sigma-1/2}}\right)\right)\right).$$
Furthermore, if $p>Y$ then $|z|<p^{\sigma}$ and hence
$$ \ex\left(\exp\left(z\tmop{Re} \frac{X(p)}{p^{\sigma}}\right)\right)= 
\ex\left(1+z\tmop{Re} \frac{X(p)}{p^{\sigma}} +O\left(\frac{|z|^2}{p^{2\sigma}}\right)\right)
= 1+O\left(\frac{|z|^2}{p^{2\sigma}}\right).$$
The independence of the $X(p)$'s together with the fact that $\sum_{p>Y}p^{-2\sigma}\ll Y^{1-2\sigma}$ imply that
\begin{equation}\label{trunc3}
\ex\bigg(\exp\Big(z\tmop{Re} \log\zeta(\sigma, X)\Big)\bigg)= \ex\left(\exp\left(z\tmop{Re} \big(R_Y(\sigma,X)\big)+ O\left(\frac{|z|}{Y^{\sigma-1/2}}\right)\right)\right),
\end{equation}
from which \eqref{trunc1} follows. To obtain \eqref{trunc2} one also uses that
$$\ex\bigg(\exp\Big(z \tmop{Im} \log\zeta(\sigma, X)\Big)\bigg)= \ex\left(\exp\left(z \tmop{Im} \big(R_Y(\sigma,X)\big)+ O\left(\frac{|z|}{Y^{\sigma-1/2}}\right)\right)\right),
$$
which can be obtained along similar lines.

\end{proof}

\begin{proof}[Proof of Theorem \ref{characteristic thm}]
Let $Y=(\log T)^{B/(\sigma-1/2)}$ where $B=B(A)$ is a suitably large constant that will be chosen later. Then it follows from Lemma \ref{gs formula} that 
\begin{equation}\label{approximation poly}
\log\zeta(\sigma+it)=R_Y(\sigma+it)
+O\left(\frac{1}{(\log T)^{B/2-3}}\right),
\end{equation}
for all $t\in [T, 2T]$ except a set of measure $T^{1-d(\sigma)}$ for some  constant $d(\sigma)>0$.
Let $\mathcal B(T)$ be the set of $t \in [T, 2T]$
such that \eqref{approximation poly} holds.
Note that $|e^{ib}-e^{ia}|=|\int_a^b e^{ix} dx|\leq|b-a|$.
Therefore, we obtain
\begin{align*}
\Phi_{\sigma,T}(u,v)&=\frac1T \int_{\mathcal B(T)} \exp\Big( i u \tmop{Re} \log \zeta(\sigma+it)
+iv \tmop{Im} \log \zeta (\sigma+it) \Big) dt  +O\left(T^{-d(\sigma)}\right)\\
&=\frac1T \int_{\mathcal B(T)}
\exp\bigg(iu \tmop{Re} R_Y(\sigma+it)
+iv \tmop{Im} R_Y(\sigma+it) \bigg)dt
+O\left(\frac{1}{(\log T)^{B/2-4}}\right)\\
&= \frac1T \int_T^{2T}
\exp\bigg(iu \tmop{Re} R_Y(\sigma+it)
+iv \tmop{Im} R_Y(\sigma+it) \bigg)dt
+O\left(\frac{1}{(\log T)^{B/2-4}}\right).
\end{align*}
Let $\mathcal{A}(T)$ be as in Proposition \ref{complex}. Then, by Lemma \ref{poly large deviation} and Proposition \ref{complex}, taking $z_1=\tfrac{i}{2}(u-iv)$ and $z_2=\tfrac{i}{2}(u+iv)$ there, we get
\begin{align*}
&\frac1T \int_T^{2T}
\exp\bigg(iu \tmop{Re} R_Y(\sigma+it)
+iv \tmop{Im} R_Y(\sigma+it) \bigg)dt\\
&= \frac1T \int_{\mathcal{A}(T)}
\exp\bigg(iu \tmop{Re} R_Y(\sigma+it)
+iv \tmop{Im} R_Y(\sigma+it) \bigg)dt
+O\left(\frac{1}{(\log T)^B}\right)\\
&= \ex\bigg(\exp\bigg(iu \tmop{Re} R_Y(\sigma,X)+ iv \tmop{Im}R_Y(\sigma,X)\bigg)\bigg) +O\left(\frac{1}{(\log T)^B}\right).\\
\end{align*}
Finally, using \eqref{trunc2} we deduce
$$\ex\bigg(\exp\bigg(iu \tmop{Re} R_Y(\sigma,X) +iv \tmop{Im} R_Y(\sigma,X)\bigg)\bigg)= \Phi_{\sigma}^{\textup{rand}}(u,v)
  +O\left(\frac{1}{(\log T)^{B-1}}\right).
$$
Choosing $B= 2A+8$, and collecting the above estimates completes the proof.
\end{proof}

\begin{proof}[Proof of Theorem \ref{complex theorem}]
As in the proof of Theorem \ref{characteristic thm}, let $Y=(\log T)^{B/(\sigma-1/2)}$ where $B=2A+8$, and 
$\mathcal B(T)$ be the set of $t \in [T, 2T]$
such that \eqref{approximation poly} holds. Then, by Lemma \ref{gs formula} $\text{meas} \big([T,2T]\setminus 
\mathcal B(T)\big)\ll T^{1-d(\sigma)}$ for some  constant $d(\sigma)>0$.
Moreover, let $\mathcal A(T)$ be as in Proposition \ref{complex}. We define
$$\mathcal E(T):= [T, 2T]\setminus\bigg(\mathcal A(T) \cap \mathcal B(T)\bigg).$$
Then, it follows from Lemma \ref{poly large deviation} that 
\begin{equation}\label{bound measure}
\text{meas} \big(\mathcal E(T)\big)\ll T\exp\left(-b_0\frac{\log T}{\log\log T}\right),
\end{equation}
for some positive constant $b_0=b_0(\sigma, A)$.

Now, by \eqref{approximation poly} we get
\begin{equation}\label{trunc moments}
\begin{aligned}
\frac{1}{T} \int_{[T,2T]\setminus\mathcal{E}(T)} |\zeta(\sigma+it)|^zdt  &= \frac{1}{T} \int_{\mathcal A(T) \cap \mathcal B(T)} \exp\left(z \tmop{Re}\big(R_Y(\sigma+it)\big) + O\left(\frac{1}{(\log T)^A}\right)\right)dt\\
&= \frac1T\int_{\mathcal A(T) \cap \mathcal B(T)} \exp\Big(z \tmop{Re} \big(R_Y(\sigma+it)\big) \Big) dt +E_4,
\end{aligned}
\end{equation}
where 
\begin{equation}\label{error moments1}
\begin{aligned}
E_4&\ll \frac{1}{T(\log T)^A}\int_{\mathcal A(T)}\exp\bigg(\tmop{Re}(z) \tmop{Re} \big(R_Y(\sigma+it)\big)\bigg) dt\\
&\ll \frac{1}{(\log T)^A}\ex\bigg(\exp\bigg(\tmop{Re}(z) \tmop{Re} \big(R_Y(\sigma,X)\big)\bigg)\bigg)\\
&\ll \frac{1}{(\log T)^A}\ex\Big(|\zeta(\sigma,X)|^{\tmop{Re}(z)}\Big),
\end{aligned}
\end{equation}
by Proposition \ref{complex} and Lemma \ref{TruncateRandom}. 

On the other hand, since $\text{meas} \big([T,2T]\setminus 
\mathcal B(T)\big)\ll T^{1-d(\sigma)}$, and $|R_Y(\sigma+it)|\leq (\log T)^{1-\sigma}/\log\log T$ for all $t\in \mathcal A(T)$ we deduce that
\begin{equation}\label{error moments2}
\begin{aligned}
\frac1T\left(\int_{\mathcal A(T)}- \int_{\mathcal A(T) \cap \mathcal B(T)}\right)\exp\Big(z \tmop{Re} \big(R_Y(\sigma+it)\big) \Big) dt
&\ll T^{-d(\sigma)} \exp\left(\tmop{Re}(z) \frac{(\log T)^{1-\sigma}}{\log \log T}\right)\\
& \ll T^{-d(\sigma)/2}.
\end{aligned}
\end{equation}
Furthermore, combining Proposition \ref{complex} and Lemma \ref{TruncateRandom} we obtain
\begin{equation}\label{main moments}
\begin{aligned}
\frac1T\int_{\mathcal A(T) } \exp\Big(z \tmop{Re} \big(R_Y(\sigma+it)\big) \Big) dt
&= \ex\bigg(\exp\Big(z \tmop{Re} \big(R_Y(\sigma,X)\big)\Big)\bigg)+O\left(\frac{1}{(\log T)^A}\right)\\
&= \ex\Big(|\zeta(\sigma, X)|^z\Big)+O\left(\frac{1}{(\log T)^A}\ex\Big(|\zeta(\sigma,X)|^{\tmop{Re}(z)}\Big)\right).
\end{aligned}
\end{equation}
The result follows upon inserting the estimates \eqref{error moments1}, \eqref{error moments2} and \eqref{main moments} in \eqref{trunc moments}.
\end{proof}

\section{$L^{2k}$ norm of $\log \zeta(\sigma + it) - a$: Proof of Proposition \ref{apoint prop}}

As a special case of Lemma 2.2.1 of Guo \cite{Guo},
which itself is a generalization of a lemma
of Landau (see \cite{Landau} or
Lemma $\alpha$ from Chapter III of \cite{Titchmarsh}), we 
have
\begin{lemma} \label{landau lem} 
Let $0< r \ll 1$. Also, let $s_0
=\sigma_0+it$ and suppose $f(z)$ is 
analytic in $|z-s_0| \leq r$. Define
\[
M_{r}(s_0)=\max_{|z-s_0| \leq r} 
\bigg|\frac{f(z)}{f(s_0)} \bigg| +3 \qquad
\mbox{ and }  \qquad N_{ r}(s_0)=\sum_{|\varrho-s_0|
\leq r} 1,
\]
where the last sum runs over the zeros, $\varrho$, of $f(z)$
in the closed disk of radius $r$ centered at $s_0$.
Then for $0<\delta<r/2$ and $|z-s_0| \leq r-2\delta$
we have
\[
\frac{f'}{f}(z)=\sum_{|\rho-s_0| \leq r-\delta}
\frac{1}{s-\rho} +O\bigg( \frac{1}{\delta^2}
\Big(\log M_{r}(s_0) +
N_{r-\delta}(s_0) (\log 1/\delta+1) \Big)\bigg).
\]
\end{lemma}

In the following we take
\begin{equation} \label{def f}
f(z)=f_a(z)=
\begin{cases}
(\zeta(z)-a)/(1-a) \mbox{ if } a\neq1, \\
2^z(\zeta(z)-1) \mbox{ if } a=1.
\end{cases}
\end{equation}
We also choose
\begin{equation} \label{def parameters}
\delta=(\sigma-1/2)/5 \qquad \mbox{and} \qquad
r=\sigma_0-(\sigma+1/2)/2,
\end{equation}
where $\sigma_0$ is taken to be 
large enough (depending on $a$) 
so that $|f(\sigma_0+it)| \geq 1/10$
and $\min_{\rho_a}|s_0-\rho_a| \geq 1/10$ uniformly in $t$.

For
$|z-s_0| \leq r-2\delta$
Lemma \ref{landau lem} yields
\begin{equation} \label{log deriv form}
\frac{\zeta'(z)}{\zeta(z)-a}=\sum_{|\rho_a-s_0| \leq r-\delta}
\frac{1}{z-\rho_a}+O\bigg(\frac{1}{\delta^2}(
\log M_{r}(s_0)+N_{r-\delta}(s_0))
(\log 1/\delta+1) \bigg).
\end{equation}

\begin{lemma} \label{first formula}
Let $1/2< \sigma \leq 2$ be fixed. Also, let 
$\delta$, $r$, and $\sigma_0$ be as in 
\eqref{def parameters}.
For $t$ sufficiently large we have
\[
\log |\zeta(\sigma+it)-a|=\sum_{|\rho_a-s_0| \leq r-\delta} 
\log |\sigma+it-\rho_a|+O(\log M_{r}(s_0)).
\]
\end{lemma}
\begin{proof}
Let $f(z)$ be as in \eqref{def f}.
First, note that Jensen's formula gives
\[
\int_0^r N_{x}(s_0) \, \frac{dx}{x}=\frac{1}{2\pi}
\int_0^{2\pi} \log |f(re^{i \theta}+s_0)| \, d \theta
-\log |f(s_0)|.
\]
Observe that
\[
\int_0^r N_{x}(s_0) \, \frac{dx}{x}
\geq \int_{r-\delta}^r N_{x}(s_0) \, \frac{dx}{x}
\geq \frac{\delta}{r} N_{r-\delta}(s_0).
\]
By this and the bound $\log |f(s_0)| \geq
\log 1/10$ we have
\begin{equation} \label{N bd}
N_{r-\delta}(s_0) 
\leq \frac{r}{\delta} \Big(\frac{1}{2\pi} \int_0^{2\pi}
 \log |f(re^{i \theta}+s_0)| \, d\theta
-\log |f(s_0)|\Big)
\leq \frac{r}{\delta} (\log M_{r}(s_0)+
\log 10).
\end{equation}
Applying this estimate in \eqref{log deriv form}
and noting that $\delta \gg 1$ and $r \ll 1$ we have
\[
\frac{\zeta'(z)}{\zeta(z)-a}=\sum_{|\rho_a-s_0| \leq r-\delta}
\frac{1}{z-\rho_a}+O(\log M_{r}(s_0))
\]
for $|z-s_0|\leq r-2\delta$. In particular, this formula is 
valid along the line segment that connects $s$ to $s_0$. 
Hence, integrating the above equation from $s$ to $s_0$
and taking real parts gives
\[
\log |\zeta(s)-a|-\log |\zeta(s_0)-a|
=\sum_{|\rho_a-s_0|\leq r-\delta} (\log |s-\rho_a|
-\log |s_0-\rho_a|)+O(\log M_{r}(s_0)).
\]
By the choice of $\sigma_0$ we have
\[
\log |\zeta(s_0)-a|=O(1) \qquad \mbox{and} \qquad
 \log |s_0-\rho_a|=O(1).
\]
Thus,
\[
\log |\zeta(s)-a|=\sum_{|\rho_a-s_0|\leq r-\delta} 
\log |s-\rho_a|+O(N_{r-\delta}(s_0)+\log M_{r}(s_0)).
\]
Applying \eqref{N bd} to the error term completes
the proof.
\end{proof}

\begin{lemma} \label{sum bound}
Let $\tfrac12 < \sigma \leq 2$ be fixed.
Also, let $r$, $\delta$, and
$\sigma_0$ be as in \eqref{def parameters}.
 Then there exists
an absolute constant $C>0$
such that for any real number $k\geq \tfrac12$
we have
\[
\frac1T \int_T^{2T} \bigg(\sum_{|\rho_a-s_0| \leq r-\delta} 
|\log |\sigma+it-\rho_a|| \bigg)^{2k} dt \ll
\Gamma(2k+1) (C\log M_{r+\delta}(s_0))^{2k}.
\]
\end{lemma}

\begin{proof}
Define $D_R(z)$ to be the closed
disc of radius $R$ centered at $z$. For
$n= \lfloor T
\rfloor, \ldots, \lfloor 2 T
\rfloor+1$
let 
\[
\mathcal D_n = \bigcup_{\ell=0}^{\lfloor 1/ \sqrt{\delta} \rfloor+1} 
D_{r}(\sigma_0+i(n+\ell \cdot \sqrt{\delta}  )).
\]  
Observe that
\[
D_{r-\delta}(\sigma_0+in) 
\bigcup \Big\{ z: n \leq \tmop{Im} z \leq n+\sqrt{\delta}, 
\sigma_0-(r-\delta) \leq \tmop{Re} z \leq \sigma_0+r-\delta \Big\} \subset D_{r}(\sigma_0+in).
\]
Next, note that
\[
\{ z: n+\sqrt{\delta} \leq \tmop{Im} z \leq n+2\sqrt{\delta}, 
\sigma_0-(r-\delta) \leq \tmop{Re} z \leq \sigma_0+r-\delta \}
\subset D_{r}(\sigma_0+i(n+\sqrt{\delta})),
\]
and so on. Hence, by construction
\[
\bigcup_{n \leq t \leq n+1} D_{r-\delta}(\sigma_0+it) \subset \mathcal D_n.
\]
This implies that
\begin{equation} \notag
\begin{split}
\int_T^{2T} \bigg(\sum_{|\rho_a-s_0| \leq  r-\delta} 
|\log |\sigma+it-\rho_a|| \bigg)^{2k} dt
\leq&
\sum_{n=\lfloor T \rfloor}^{\lfloor 2T \rfloor+1}
\int_{n}^{n+1} \bigg(\sum_{|\rho_a-s_0| \leq r-\delta} 
|\log |\sigma+it-\rho_a|| \bigg)^{2k} dt \\
\leq &
 \sum_{n=\lfloor T \rfloor}^{\lfloor 2T \rfloor+1} 
\int_{n}^{n+1} 
\bigg(\sum_{\rho_a \in \mathcal D_n} 
|\log |\sigma+it-\rho_a|| \bigg)^{2k} dt.
\end{split}
\end{equation}
Applying Minkowski's inequality to the right-hand side
we get that
\begin{equation} \label{minkowski bd}
\begin{split}
&\int_T^{2T} \bigg(\sum_{|\rho_a-s_0| \leq  r-\delta} 
|\log |\sigma+it-\rho_a|| \bigg)^{2k} dt  \\
&\qquad \qquad \qquad \qquad \qquad \qquad 
 \leq \sum_{n=\lfloor T \rfloor}^{\lfloor 2T \rfloor+1}
\bigg(\sum_{\rho_a \in \mathcal D_n} \bigg(\int_{n}^{n+1} 
|\log |\sigma+it-\rho_a| |^{2k} dt\bigg)^{1/{(2k)}} \bigg)^{2k}.
\end{split}
\end{equation}

We now estimate the inner integral on the right-hand side.
We have for $n \leq t \leq n+1$
and $\rho_a \in \mathcal D_n$ that
\[
|t-\gamma_a|\leq |\sigma+it-\rho_a| \leq c
\] 
for some absolute constant $c=c(a) > 1$. So
for $n \leq t \leq n+1$
and $\rho_a \in \mathcal D_n$ we get that
\begin{equation} \label{max bd}
|\log |\sigma+it-\rho_a||^{2k} 
\leq |\log |t-\gamma_a||^{2k}
+|\log c|^{2k}.
\end{equation}
Also, for $\rho_a \in \mathcal D_n$ we have 
$n-r \leq \gamma_a \leq n+r+2$.
Thus,
\begin{equation} \label{simple bd}
\begin{split}
\int_n^{n+1}| \log |t-\gamma_a||^{2k} dt
\leq& \int_{n-r}^{n+r+2} |\log |t-\gamma_a||^{2k} dt\\
\leq& 2 \int_0^{2r+2}|\log x|^{2k} dx\\
=&2 \Gamma(2k+1)+O\Big( (\log (2r+2))^{2k}\Big).
\end{split}
\end{equation}

Next, note that the set $\mathcal D_n$ consists of $\ll 1/\sqrt{\delta}
=O(1)$ 
disks, each of radius $r$.
Arguing as in \eqref{N bd}, 
we see that each one contains
$\ll \delta^{-1} \log M_{r+\delta}(s_0)
\ll \log M_{r+\delta}(s_0)$ zeros. Hence,
by this, \eqref{minkowski bd}, \eqref{max bd},
and \eqref{simple bd} we see that
\begin{equation} \notag
\begin{split}
\int_T^{2T} \bigg(\sum_{|\rho_a-s_0| \leq  r-\delta} 
|\log |\sigma+it-\rho_a|| \bigg)^{2k} dt
\ll& \sum_{n=\lfloor T \rfloor}^{\lfloor 2T \rfloor+1}
\Gamma(2k+1)\bigg(\sum_{\rho_a \in \mathcal D_n} 1\bigg)^{2k}\\
\leq &T \, \Gamma(2k+1) (C \log M_{r+\delta}(s_0))^{2k},
\end{split}
\end{equation}
for some absolute constant $C>0$.
\end{proof}

\begin{lemma} \label{mean square}
Let $\tfrac12< \sigma \leq 2$ be fixed. 
For any fixed $\sigma_0 > 1$ and $R=\sigma_0-\sigma$ we have
\[
\int_T^{2T} (M_{R}(s_0))^2 \, dt \ll T.
\]
\end{lemma}
\begin{proof}
First of all,
\[
\int_T^{2T} (M_{R}(s_0))^2 \, dt \leq \sum_{n=\lfloor T \rfloor}^{\lfloor 2T \rfloor+1} 
\int_n^{n+1} (M_{R}(s_0))^2 \, dt.
\]
Next, let $D_R(z)$ be the disk of radius $R$ centered at $z$.
Also, let
$s_n=\sigma_n+it_n$
be a point at which $|\zeta(s)|$ achieves
its maximum value on the 
set $ \cup_{n \leq t \leq n+1} D_{R}(s_0)$.
Thus,
\[
\int_n^{n+1} (M_{R}(s_0)) ^2\, dt \ll |\zeta(s_n)|^2+1.
\]
Hence, we have
\begin{equation} \label{M bd}
\int_T^{2T} (M_{R}(s_0))^2 \, dt \ll \sum_{n=\lfloor T \rfloor}^{\lfloor 2T \rfloor+1}  |\zeta(s_n)|^2+T.
\end{equation}
Let $R'=\sigma_0-(\sigma+1/2)/2$.
To bound $|\zeta(s_n)|^2$ we note that
\begin{equation} \label{int bd}
|\zeta(s_n)|^2 \leq \frac{4}{\pi(\sigma-\frac12)^2} \iint_{D_{R'}(\sigma_0+it_n)} |\zeta(x+iy)|^2 \, dx \, dy.
\end{equation}
(For a proof of this inequality see the lemma preceding Theorem 11.9 of Titchmarsh \cite{Titchmarsh}).

Let $\mathcal S_j=\{s_n : n \equiv j \pmod{(4 \lceil R' \rceil+2)} \}$.
If $s_m, s_n \in \mathcal S_j$ and $m \neq n$
then $|m-n| \geq 4 \lceil R' \rceil +2$; so that $|t_m-t_n|\ge2R'+1$. This implies that
$D_{R'}(\sigma_0+it_n) \cap D_{R'}(\sigma_0+it_m) =\emptyset$. Thus, since the disks are disjoint
we see that by \eqref{int bd} we have
\[
\sum_{s_n \in S_j} |\zeta(s_n)|^2 \ll 
\int\limits_{\frac12\cdot(\sigma+\frac12)}^{2\sigma_0-\frac12\cdot(\sigma+\frac12)} 
\int_{T-R'-R-1}^{2T+R'+R+1} |\zeta(u+it)|^2 \, dt \, du.
\]
Applying, the well-known mean value estimate for $\zeta(s)$ to the inner integral
(see Theorem 7.2(A) of \cite{Titchmarsh}) we have (uniformly in $j$)
\[
\sum_{s_n \in S_j} |\zeta(s_n)|^2 \ll T.
\]
Also,
$
\{s_n\}=\coprod_{j} \mathcal S_j
$. Thus,
\[
\sum_{n=\lfloor T \rfloor}^{\lfloor 2T \rfloor+1} |\zeta(s_n)|^2=
\sum_{j=0}^{4 \lceil R' \rceil+1} \sum_{s_n \in \mathcal S_j} |\zeta(s_n)|^2 \ll T.
\]
Inserting this into \eqref{M bd} completes the proof.
\end{proof}
\begin{lemma} \label{M bound}
Let $\tfrac12 < \sigma \leq 2$ be fixed. Also, let $r$, $\delta$, and $\sigma_0$ be
as in \eqref{def parameters}. 
Then, there exists an absolute constant 
$C>0$ such that for any real number $k\geq 1$
\[
\int_T^{2T} (\log M_{r+\delta} (s_0))^{2k}  dt \ll T (C k)^{2k}.
\]
\end{lemma}

\begin{proof}
Let $f(x)=(\log (x+e^{2k-1}))^{2k}$, where $k \geq 1$. Note that $f''(x)<0$
for $x>0$. Thus, by Jensen's inequality
and Lemma \ref{mean square} we have
\begin{equation*}
\begin{split}
\frac1T \int_T^{2T} (\log M_{r+\delta} (s_0))^{2k}  dt <&
\frac{1}{4^k T}\int_T^{2T} (\log ((M_{r+\delta} (s_0))^2+e^{2k-1}))^{2k}  \, dt\\
\leq&
 \frac{1}{4^k} \Bigg(\log\bigg(\frac1T \int_T^{2T} (M_{r+\delta} (s_0))^2 \,dt+e^{2k-1}\bigg)\Bigg)^{2k}\\
\ll& (Ck)^{2k},
\end{split}
\end{equation*}
for some absolute constant $C$.
\end{proof}
\begin{proof}[Proof of Proposition \ref{apoint prop}]
First we consider the case $k \geq 1$. Note that
by Lemma \ref{first formula} we have
\begin{equation} \notag
\begin{split}
\int_T^{2T}  |\log |\zeta(\sigma+it)-a||^{2k} 
\, dt
 \leq& 4^k \int_T^{2T} \bigg( \sum_{|\rho_a-s_0| \leq r-\delta} 
|\log |\sigma+it-\rho_a||\bigg)^{2k} \, dt\\
&+O\bigg(  4^k \int_T^{2T} (\log M_{r}(s_0))^{2k} \, dt \bigg).
\end{split}
\end{equation}
Hence, for this case, we see that Proposition \ref{apoint prop} follows 
from the above inequality, Lemma \ref{sum bound}, 
and Lemma \ref{M bound}. For $0<k<1$ the proposition
follows from an application of H\"older's inequality.
\end{proof}

\section{Bounding the discrepancy: Proof of Theorem \ref{discrepancy thm}}

In order to prove Theorem \ref{discrepancy thm} we shall appeal to the following Lemma of Selberg (Lemma 4.1 of \cite{Tsang}), which provides a smooth approximation for the signum function. Selberg used this lemma in his proof that $\log\zeta(\tfrac12+it)$ has a limiting two-dimensional Gaussian distribution (see \cite{Tsang} and \cite{SelbergPaper}). Recall that the signum function is defined by
$$ \textup{sgn}(x)=\begin{cases} -1 & \text{ if } x<0, \\ 0 & \text{ if } x=0, \\ 1 & \text{ if } x>0.\end{cases}$$

\begin{lemma}[Selberg, Lemma 4.1 of \cite{Tsang}] \label{Selberg}
Let $L>0$. Define
$$G(u)=\frac{2u}{\pi}+2(1-u)u\cot(\pi u) \quad \text{ for } u\in [0,1].$$
Then for all $x\in \mathbb{R}$ we have
$$ \textup{sgn}(x)=  \int_0^LG\left(\frac{u}{L}\right)\sin(2\pi  ux)\frac{du}{u} + O\left(\left(\frac{\sin(\pi Lx)}{\pi L x}\right)^2\right).$$
Moreover, $G(u)$ is differentiable and $0\leq G(u)\leq 2/\pi$ for $u\in[0,1]$.
\end{lemma}

For any rectangle $\mathcal{R}$ in the complex plane, let $\mathbf{1}_{\mathcal{R}}$ denote its indicator function. Using Lemma \ref{Selberg} we derive a smooth approximation for $\mathbf{1}_{\mathcal{R}}$ which will be used to prove Theorem \ref{discrepancy thm}. For any $\alpha, \beta\in \mathbb{R}$, we define
$$ f_{\alpha,\beta}(u):=\frac{e^{-2\pi i \alpha u}-e^{-2\pi i \beta u}}{2}.$$
Then, we prove
\begin{lemma}\label{ApproxRectangle}
Let $
\R=\{z=x+iy \in \mathbb{C}: a_1< x< a_2 \text{ and } b_1< y< b_2\},
$ and $L>0$ be a real number. For any $z=x+iy\in \mathbb{C}$ we have
\begin{align*}
\mathbf{1}_{\mathcal{R}}(z)=  W_{L, \mathcal{R}}(z) + & O\Big(\frac{\sin^2(\pi L(x-a_1))}{(\pi L (x-a_1))^2}+ \frac{\sin^2(\pi L(x-a_2))}{(\pi L (x-a_2))^2}\\
&+\frac{\sin^2(\pi L(y-b_1))}{(\pi L (y-b_1))^2}+ \frac{\sin^2(\pi L(y-b_2))}{(\pi L (y-b_2))^2}\Big)\\
\end{align*}
where $W_{L, \mathcal{R}}(z)$ equals
$$\frac12\textup{Re}\int_0^L\int_0^LG\left(\frac{u}{L}\right)G\left(\frac{v}{L}\right)\left(e^{2\pi i(ux-vy)}f_{a_1,a_2}(u)\overline{f_{b_1,b_2}(v)}-e^{2\pi i(ux+vy)}f_{a_1,a_2}(u)f_{b_1,b_2}(v)\right)\frac{du}{u}\frac{dv}{v}.
$$
\end{lemma}
\begin{proof}
Here and throughout we shall denote by $\mathbf{1}_{\alpha,\beta}$ the indicator function of the interval $(\alpha,\beta)$. Observe that
$$ \mathbf{1}_{\alpha,\beta}(x)=\frac{\textup{sgn}(x-\alpha)-\textup{sgn}(x-\beta)}{2}+O\big(\delta(x-\alpha)+\delta(x-\beta)\big),$$
where $\delta(x)$ is the Dirac delta function (it equals $1$ when $x=0$, and zero otherwise).

Furthermore, it follows from Lemma \ref{Selberg} that
\begin{equation}\label{ApproxInterval}
\mathbf{1}_{\alpha,\beta}(x)= \textup{Im} \int_0^L G\left(\frac{u}{L}\right)e^{2\pi i ux}f_{\alpha,\beta}(u)\frac{du}{u} + O\left(\frac{\sin^2(\pi L(x-\alpha))}{(\pi L (x-\alpha))^2}+ \frac{\sin^2(\pi L(x-\beta))}{(\pi L (x-\beta))^2}\right).
\end{equation}
The result follows from the fact that $\mathbf{1}_{\mathcal{R}}(z)= \mathbf{1}_{a_1,a_2}(x)\mathbf{1}_{b_1,b_2}(y)$ together with \eqref{ApproxInterval} and  the identity
\begin{equation}\label{ComplexIdentity}
\textup{Im}(w_1)\textup{Im}(w_2)= \frac{1}{2}\textup{Re}(w_1\overline{w_2}-w_1w_2).
\end{equation}
\end{proof}
The last ingredient we need in order to establish Theorem \ref{discrepancy thm} is the following lemma.
\begin{lemma}\label{FourierDecay}
Let $\tfrac12 <\sigma\leq 1$. Let $u$ be a large positive real number, then
$$ \ex\Big(\exp\Big(i u \tmop{Re}\log \zeta(\sigma,X)\Big) \Big)\ll \exp\left(-\frac{u}{5\log u}\right),$$
and
$$ \ex\Big(\exp\Big(i u \tmop{Im}\log \zeta(\sigma,X)\Big) \Big)\ll \exp\left(-\frac{u}{5\log u}\right).$$
\end{lemma}
\begin{proof}
First, note that $\ex(e^{is \tmop{Re} X(p)})=\ex(e^{is \tmop{Im} X(p)})=J_0(s)$ for all $s\in \mathbb{R}$ and all primes $p$,  where $J_0(s)$ is the Bessel function of order $0$. We shall prove only the first inequality since the second can be derived similarly. We have
$$ \ex\Big(\exp\Big(iu \tmop{Re}\log \zeta(\sigma,X)\Big) \Big)= \prod_{p} \ex\left(\exp\left(-iu \tmop{Re}\log\left(1-\frac{X(p)}{p^{\sigma}}\right)\right)\right).$$
Therefore, we deduce that
\begin{align*} \left|\ex\Big(\exp\Big(iu \tmop{Re}\log \zeta(\sigma,X)\Big) \Big)\right| &\leq \prod_{\sqrt{u}\leq p\leq  u}\ex\left(\exp\left(\frac{i u}{p^{\sigma}}\tmop{Re} X(p) +O\left(\frac{u}{p^{2\sigma}}\right)\right)\right)\\
&= \exp\left(O\left(u^{3/2-\sigma}\right)\right)\prod_{\sqrt{u}\leq p\leq u/2}J_0\left(\frac{u}{p^{\sigma}}\right).
\end{align*}
Now, using that $|J_0(x)|\leq e^{-1/2}$ for all $x\geq 2$, along with the prime number theorem we obtain
$$
 \left|\ex\Big(\exp\Big(iu \tmop{Re}\log \zeta(\sigma,X)\Big) \Big)\right| \ll \exp\left(-\frac{1}{2}\pi(u/2)+ O\left(u^{3/2-\sigma}\right)\right)\ll \exp\left(-\frac{u}{5\log u}\right),
$$
as desired.
\end{proof}
\begin{proof}[Proof of Theorem \ref{discrepancy thm}]
We only consider the case where $\tfrac12<\sigma<1$, since the analogous result for $\sigma=1$ can be obtained along similar lines. To shorten our notation we let
$$\Psi_T(\R)= \mathbb{P}_{T} \Big (\log \zeta(\sigma + it) \in \mathcal R \Big ), \textup{ and } \Psi(\R)= \mathbb P\Big( \log \zeta(\sigma, X) \in \mathcal R \Big).$$

Let $\R$ be  a rectangle with  sides parallel to the coordinate axes, and $\widetilde{\R}= \R\cap  [-\log\log T, \log\log T]\times [-\log\log T, \log\log T]$. Then using the large deviation result \eqref{large deviation asymptotic}  we deduce that
$$ \Psi_T\big(\R\big)= \Psi_T\big(\widetilde{\R}\big)+O\left(\frac{1}{(\log T)^2}\right).$$  
Similarly, one has 
$$ \Psi(\R)= \Psi\big(\widetilde{\R}\big)+O\left(\frac{1}{(\log T)^2}\right).$$
Let $\s$ be the set of rectangles $\R \subset  [-\log\log T, \log\log T]\times [-\log\log T, \log\log T]$ with  sides parallel to the coordinate axes. Then, we deduce that 
$$ D_{\sigma}(T)= \sup_{\R\in \s} |\Psi_T(\R)-\Psi(\R)| + O\left(\frac{1}{(\log T)^2}\right).$$

Let  $\R$ be a rectangle in $\s$ and $L$  a positive real number to be chosen later. Then it follows from Lemma \ref{ApproxRectangle}  that
\begin{equation}\label{distribution}
\Psi_T\big(\R\big)= \frac{1}{T}\int_T^{2T} W_{L, \R}\big(\log\zeta(\sigma+it)\big) dt +O\Big(I_{T}(L,a_1)+ I_{T}(L,a_2)+J_{T}(L,b_1)+J_{T}(t,b_2)\Big)
\end{equation}
where
$$I_{T}(L,s)= \frac{1}{T}\int_{T}^{2T}\frac{\sin^2\big(\pi L(\tmop{Re} \log\zeta(\sigma+it)-s)\big)}{(\pi L(\tmop{Re} \log\zeta(\sigma+it)-s))^2}dt,$$
and
$$ J_{T}(L,s)= \frac{1}{T}\int_{T}^{2T}\frac{\sin^2\big(\pi L(\tmop{Im} \log\zeta(\sigma+it)-s)\big)}{(\pi L(\tmop{Im} \log\zeta(\sigma+it)-s))^2}dt.$$
It follows from Theorem \ref{characteristic thm} that there exists a positive constant $c=c(\sigma)$ such that for all $|u|, |v|\leq c(\log T)^{\sigma}$ we have 
\begin{equation}\label{characteristic}
\Phi_{\sigma,T}(u, v)= \Phi_{\sigma}^{\textup{rand}}(u,v)+ O\left(\frac{1}{(\log T)^5}\right).
\end{equation}
First, we handle the main term of \eqref{distribution}
\begin{equation}\label{MainTerm}
\begin{aligned}
\frac{1}{T}\int_T^{2T} W_{L, \R}\big(\log\zeta(\sigma+it)\big) dt
=\frac12\textup{Re}\int_0^L\int_0^L&G\left(\frac{u}{L}\right)G\left(\frac{v}{L}\right)\Big(\Phi_{\sigma,T}(2\pi u, -2\pi v)f_{a_1,a_2}(u)\overline{f_{b_1,b_2}(v)}\\
&-\Phi_{\sigma,T}(2\pi u, 2\pi v)f_{a_1,a_2}(u)
f_{b_1,b_2}(v)\Big)\frac{du}{u}\frac{dv}{v}.
\end{aligned}
\end{equation}
We choose $L= c(\log T)^{\sigma}$. Then inserting the estimate \eqref{characteristic} in equation \eqref{MainTerm} and using that
\begin{equation}\label{boundf}
|f_{\alpha,\beta}(u)|=\frac{1}{2}\left|\int_{2\pi \alpha u}^{2\pi \beta u} e^{-it}dt\right|\leq \pi u |\beta-\alpha|,
\end{equation}
we obtain
\begin{equation}\label{MainTerm2}
\begin{aligned}
\frac{1}{T}\int_T^{2T} W_{L, \R}\big(\log\zeta(\sigma+it)\big) dt
= & \ex\left(W_{L, \R}\big(\log \zeta(\sigma, X)\big)\right)+ O\left(\text{meas}_2(\R)\frac{L^2}{(\log T)^5}\right)\\
= & \ex\left(W_{L, \R}\big(\log \zeta(\sigma, X)\big)\right)+ O\left(\frac{1}{(\log T)^2}\right),\\
\end{aligned}
\end{equation}
where $\text{meas}_2$ denotes the two-dimensional Lebesgue measure.
Furthermore we infer from Lemma \ref{ApproxRectangle}
\begin{equation}\label{MainTerm3}
\begin{aligned}
\ex\big(W_{L, \R}\big(\log \zeta(\sigma, X)\big)\big)= &\ex\left(\mathbf{1}_{\mathcal{R}}\big(\log \zeta(\sigma, X)\big)\right) \\
& +O\Big(I_{\text{rand}}(L,a_1)+ I_{\text{rand}}(L,a_2)+J_{\text{rand}}(L,b_1)+J_{\text{rand}}(L,b_2)\Big),\\
\end{aligned}
\end{equation}
where
$$I_{\text{rand}}(L,s)= \ex\left(\frac{\sin^2\big(\pi L(\tmop{Re} \log \zeta(\sigma, X)-s)\big)}{(\pi L(\tmop{Re} \log \zeta(\sigma, X)-s))^2}\right),$$
and
$$ J_{\text{rand}}(L,s)= \ex\left(\frac{\sin^2\big(\pi L(\tmop{Im} \log \zeta(\sigma, X)-s)\big)}{(\pi L(\tmop{Im} \log \zeta(\sigma, X)-s))^2}\right).$$
Note that $\ex\left(\mathbf{1}_{\mathcal{R}}\big(\log \zeta(\sigma, X)\big)\right)=\pr\left(\log \zeta(\sigma,X)\in \R\right)$. Moreover, in order to bound $I_{\text{rand}}(L,s)$ and $J_{\text{rand}}(L,s)$ we use the following identity
\begin{equation}\label{trigo}
 \frac{\sin^2(\pi L x)}{(\pi Lx)^2}= \frac{2(1-\cos(2\pi Lx))}{L^2(2\pi x)^2}= \frac{2}{L^2}\int_0^L(L-v) \cos(2\pi xv)dv.
\end{equation}
Indeed, using \eqref{trigo} along with Lemma \ref{FourierDecay} we obtain
\begin{equation}\label{error1}
\begin{aligned}
 I_{\text{rand}}(L,s) &=  \ex\left(\tmop{Re}\int_0^L \frac{2(L-v)}{L^2}\exp\Big(2\pi iv \big(\tmop{Re} \log \zeta(\sigma,X)-s\big)\Big) dv\right)\\
 &= \tmop{Re} \int_{0}^L\frac{2(L-v)}{L^2} e^{-2\pi i v s} \Phi_{\sigma}^{\textup{rand}}(2\pi v, 0) dv\\
 &\ll \frac{1}{L} \left(1+\int_2^L \exp\left(-\frac{v}{\log v}\right) dv \right)\\
 & \ll \frac{1}{L}.
\end{aligned}
\end{equation}
uniformly for all $s\in \mathbb{R}$. Similarly, one obtains that $J_{\text{rand}}(L,s)\ll 1/L$. Therefore, inserting these estimates in \eqref{MainTerm3} and using \eqref{MainTerm2} we deduce
\begin{equation}\label{MainTerm4}
\frac{1}{T}\int_T^{2T} W_{L, \R}\big(\log\zeta(\sigma+it)\big) dt= \pr\left(\log \zeta(\sigma,X)\in \R\right) +O\left(\frac{1}{L}\right).
\end{equation}

Now it remains to bound the error term on the right hand side of \eqref{distribution}. Using the identity \eqref{trigo} along with equations \eqref{characteristic} and \eqref{error1} we obtain
\begin{equation*}
\begin{aligned}
 I_T(L,s) &= \tmop{Re} \frac{1}{T}\int_T^{2T}\int_0^L \frac{2(L-v)}{L^2}\exp\Big(2\pi iv \big(\tmop{Re} \log\zeta(\sigma+it)-s\big)\Big) dvdt\\
 &= \tmop{Re} \int_{0}^L\frac{2(L-v)}{L^2} e^{-2\pi i v s} \Phi_{\sigma,T}(2\pi v, 0) dv\\
 &= \tmop{Re} \int_{0}^L\frac{2(L-v)}{L^2} e^{-2\pi i v s} \Phi_{\sigma}^{\textup{rand}}(2\pi v, 0) dv +O\left(\frac{1}{(\log T)^5}\right)\\
 & \ll \frac{1}{L},
\end{aligned}
\end{equation*}
uniformly for all $s\in \mathbb{R}$.
Moreover, the bound $J_T(L,s)\ll 1/L$ can be obtained along the same lines. Therefore, combining these estimates with \eqref{distribution} and \eqref{MainTerm4} we deduce
$$ \Psi_T(\R)= \Psi(\R)+O\left(\frac{1}{(\log T)^{\sigma}}\right),$$
which completes the proof.
\end{proof}

\section{Large deviations: Proof of Theorem \ref{large thm}}

For $z\in \mathbb{C}$ we define
$$ M(z)= \log \ex(|\zeta(\sigma,X)|^z).$$
Further, let $\kappa$ be the unique positive solution to the equation
$M'(k)=\tau.$ One of the main ingredients in the proof of Theorem \ref{large thm} is the following proposition which is established using the saddle-point method.
\begin{proposition}\label{SaddlePoint}
Let $\tfrac12 < \sigma<1$.
Uniformly for $\tau\geq 1$ we have
$$\pr(\log|\zeta(\sigma,X)|>\tau)= \frac{\ex\left(|\zeta(\sigma, X)|^{\kappa}\right)e^{-\tau \kappa}}{k \sqrt{2\pi M''(\kappa)}}\left(1+O\left(\kappa^{1-\frac{1}{\sigma}}\log \kappa \right)\right).$$
\end{proposition}

\subsection{Preliminaries}

Let $\chi(y)=1$ if $y>1$ and be equal to $0$ otherwise. Then we have the following smooth analogue of Perron's formula, which is a slight variation of a formula of Granville and Soundararajan (see \cite{GranvilleSound}).

\begin{lemma}\label{SmoothPerron}
Let $\lambda>0$ be a real number and $N$ be a positive integer. For any $c>0$ we have for $y>0$
$$
0\leq \frac{1}{2\pi i}\int_{c-i\infty}^{c+i\infty} y^s \left(\frac{e^{\lambda s}-1}{\lambda s}\right)^N \frac{ds}{s} -\chi(y)\leq 
\frac{1}{2\pi i}\int_{c-i\infty}^{c+i\infty} y^s \left(\frac{e^{\lambda s}-1}{\lambda s}\right)^N \frac{1-e^{-\lambda N s}}{s}ds.
$$
\end{lemma}

\begin{proof}
For any $y>0$ we have 
$$\frac{1}{2\pi i}\int_{c-i\infty}^{c+i\infty} y^s \left(\frac{e^{\lambda s}-1}{\lambda s}\right)^N\frac{ds}{s}
= \frac{1}{\lambda^N}\int_{0}^{\lambda}\cdots \int_0^{\lambda} \frac{1}{2\pi i} \int_{c-i\infty}^{c+i\infty}
\left(ye^{t_1+ \cdots+ t_n}\right)^s\frac{ds}{s} dt_1\cdots dt_N
$$
so that by Perron's formula we obtain
$$ 
\frac{1}{2\pi i}\int_{c-i\infty}^{c+i\infty} y^s \left(\frac{e^{\lambda s}-1}{\lambda s}\right)^N\frac{ds}{s}
= \begin{cases} = 1 & \text{ if } y\geq 1, \\ \in [0,1] & \text{ if } e^{-\lambda N } \leq y < 1,\\
 =0 & \text{ if } 0<y< e^{-\lambda N }. \end{cases}
$$
Therefore we deduce that 
\begin{equation}\label{indicator}
 \frac{1}{2\pi i}\int_{c-i\infty}^{c+i\infty} y^s e^{-\lambda N s} \left(\frac{e^{\lambda s}-1}{\lambda s}\right)^N \frac{ds}{s} \leq \chi(y)\leq \frac{1}{2\pi i}\int_{c-i\infty}^{c+i\infty} y^s \left(\frac{e^{\lambda s}-1}{\lambda s}\right)^N \frac{ds}{s}
\end{equation}
which implies the result.

\end{proof}

\begin{lemma}\label{Decay} Let $s=k+it$ where $k$ is a large positive real number. Then, in the range $|t|\geq k$ we have 

$$\ex\left(|\zeta(\sigma, X)|^s\right)\ll \exp\left(-|t|^{1/\sigma-1}\right) \ex\left(|\zeta(\sigma, X)|^k\right).$$

\end{lemma}

\begin{proof}
For simplicity we suppose that $t>0$. First, note that
$$ \ex\left(|\zeta(\sigma, X)|^s\right)= \prod_{p}\ex\left(\left|1-\frac{X(p)}{p^{\sigma}}\right|^{-s}\right).$$
Therefore, for any $y\geq 2$ we have 
\begin{equation}\label{decay1}
\frac{\left|\ex\left(|\zeta(\sigma, X)|^s\right)\right|}{\ex\left(|\zeta(\sigma, X)|^k\right)}\leq \prod_{p>y}\frac{\left|\ex\left(\left|1-\frac{X(p)}{p^{\sigma}}\right|^{-k-it}\right)\right|}{\ex\left(\left|1-\frac{X(p)}{p^{\sigma}}\right|^{-k}\right)}
\end{equation}
Moreover, for $p>|s|^{1/(2\sigma)}$ we have
\begin{equation}\label{Bessel1}
\ex\left(\left|1-\frac{X(p)}{p^{\sigma}}\right|^{-s}\right)= \ex\left(\left(1-2\frac{\text{Re}X(p)}{p^{\sigma}}+ \frac{1}{p^{2\sigma}}\right)^{-s/2}\right)
= I_0\left(\frac{s}{p^{\sigma}}\right)\left(1+O\left(\frac{|s|}{p^{2\sigma}}\right)\right),
\end{equation}
where $I_0(z):=\sum_{n=0}^{\infty} (z/2)^{2n}/n!^2$ is the modified Bessel function of order $0$.
Let $y=t^{2/\sigma}.$ since $I_0(z)=1+z^2/4+ O(|z|^4)$ for $|z|\leq 1$, we deduce that for all primes $p>y$ 
$$ 
\frac{\ex\left(\left|1-\frac{X(p)}{p^{\sigma}}\right|^{-s}\right)}{\ex\left(\left|1-\frac{X(p)}{p^{\sigma}}\right|^{-k}\right)}
= \exp\left(\frac{s^2-k^2}{4p^{2\sigma}}+ O\left(\frac{t}{p^{2\sigma}}+\frac{t^4}{p^{4\sigma}}\right)\right).
$$ 
Since $\text{Re}(s^2-k^2)=-t^2$, it follows from the prime number theorem and equation \eqref{decay1} that
\begin{align*} 
\frac{\left|\ex\left(|\zeta(\sigma, X)|^s\right)\right|}{\ex\left(|\zeta(\sigma, X)|^k\right)}
&\leq \exp\left(-\frac{t^2}{4}\sum_{p>y}\frac{1}{p^{2\sigma}}+ O\left(t\sum_{p>y}\frac{1}{p^{2\sigma}}+t^4\sum_{p>y}\frac{1}{p^{4\sigma}}\right)\right)\\
&\leq \exp\left(-c(\sigma)\frac{t^{2/\sigma-2}}{\log t} + O\left(t^{2/\sigma-3}\right)\right),
\end{align*}
for some constant $c(\sigma)>0$. This implies the result.

\end{proof}

Let $f(u):=\log I_0(u).$ Then, a classical estimate (see for example Lemma 3.1 of \cite{Lamzouriarg}) asserts that $f(u) \asymp u^2$ if $0\leq u\leq 1$ and  $f(u)\asymp u$ if $u\geq 1$. Similarly, we have the following standard estimates
\begin{lemma}\label{Bessel2} We have
\begin{align*} 
 f'(u) &\asymp \begin{cases} u & \text{ if } 0\leq u\leq 1\\ 1 & \text{ if } u\geq 1. \end{cases}\\
 f''(u) &\asymp \begin{cases} 1 & \text{ if } 0\leq u\leq 1\\ u^{-1} & \text{ if } u\geq 1. \end{cases}\\
f'''(u) &\asymp \begin{cases} u & \text{ if } 0\leq u\leq 1\\ u^{-2} & \text{ if } u\geq 1. \end{cases}\\
\end{align*}
\end{lemma}
Next, we have the following proposition from which we deduce an asymptotic formula for the saddle-point $\kappa$ in terms of $\tau$.
\begin{proposition}\label{saddle1} For large positive real numbers $k$, we have
\begin{equation}\label{Moments1}
M(k)= g_0(\sigma) \frac{k^{1/\sigma}}{\log k}\left(1+O\left(\frac{1}{\log k}\right)\right),
\end{equation}
where $$ g_0(\sigma):=\int_0^{\infty} \frac{f(u)}{u^{1/\sigma+1}}du,$$
and
\begin{equation}\label{Moments2}
M'(k)= g_1(\sigma)\frac{k^{1/\sigma-1}}{\log k}\left(1+O\left(\frac{1}{\log k}\right)\right).
\end{equation}
where 
$$ g_1(\sigma):=\int_0^{\infty} \frac{f'(u)}{u^{1/\sigma}}du.$$ Similarly we have 
$$ M''(k) \asymp_{\sigma} k^{1/\sigma-2}/\log k, \text{ and } M'''(k)\asymp_{\sigma}k^{1/\sigma-3}/\log k.$$
\end{proposition}
\begin{proof}
The first estimate \eqref{Moments1} follows from Proposition 3.2 of \cite{LamzouriLarge}. The other estimates can be proved along the same lines.
\end{proof}

\begin{corolary}\label{saddle2} Let $\tau$ be a large real number and let $\kappa$ be the solution to $M'(k)=\tau$. Then
$$ \kappa= g_2(\sigma) \tau^{\sigma/(1-\sigma)}(\log\tau)^{\sigma/(1-\sigma)}\left(1+O\left(\frac{\log\log\tau}{\log \tau}\right)\right),$$ where 
$$ g_2(\sigma)= \left(\frac{\sigma}{(1-\sigma)g_1(\sigma)}\right)^{\sigma/(1-\sigma)}.$$
\end{corolary}
Combining Proposition \ref{SaddlePoint}, Proposition \ref{saddle1} and Corollary \ref{saddle2} we recover the following result, which was obtained by the first author in \cite{LamzouriLarge}.
\begin{corolary}\label{EstimateRandom}
 Let $\tfrac12<\sigma<1$. There exists a constant $A(\sigma)>0$ such that uniformly for $\tau\geq 2$ we have
$$
\pr(\log |\zeta(\sigma, X)|>\tau)=\exp\left(-A(\sigma)\tau^{\frac{1}{(1-\sigma)}}(\log\tau)^{\frac{\sigma}{(1-\sigma)}}
\left(1+o(1)\right)\right).
$$
\end{corolary}
\subsection{Proof of Proposition \ref{SaddlePoint} and Theorem \ref{large thm}}

\begin{proof}[Proof of Proposition \ref{SaddlePoint}]
Let $0<\lambda<1/(2\kappa)$ be a real number to be chosen later. Using Lemma \ref{SmoothPerron} with $N=1$ we obtain
\begin{equation}\label{approximation1}
\begin{aligned}
0&\leq \frac{1}{2\pi i}\int_{\kappa-i\infty}^{\kappa+i\infty}\ex\left(|\zeta(\sigma,X)|^s\right)e^{-\tau s}\frac{e^{\lambda s}-1}{\lambda s}\frac{ds}{s}-\pr(\log |\zeta(\sigma, X)|>\tau)\\
&\leq \frac{1}{2\pi i}\int_{\kappa-i\infty}^{\kappa+i\infty} \ex\left(|\zeta(\sigma,X)|^s\right)e^{-\tau s} \frac{\left(e^{\lambda s}-1\right)}{\lambda s} \frac{\left(1-e^{-\lambda s}\right)}{s}ds.
\end{aligned}
\end{equation}
Since $\lambda\kappa<1/2$ we have
$|e^{\lambda s}-1|\leq 3 \text{ and } |e^{-\lambda s}-1|\leq 2$. 
Therefore, using Lemma \ref{Decay} we obtain
\begin{equation}\label{error12}
 \int_{\kappa-i\infty}^{\kappa-i\kappa}+ \int_{\kappa+i\kappa}^{\kappa+i\infty}\ex\left(|\zeta(\sigma,X)|^s\right)e^{-\tau s}\frac{e^{\lambda s}-1}{\lambda s}\frac{ds}{s} \ll \frac{e^{-\kappa^{1/\sigma-1}}}{\lambda \kappa} \ex\left(|\zeta(\sigma,X)|^{\kappa}\right)e^{-\tau\kappa},
\end{equation}
and similarly
\begin{equation}\label{error2}
\int_{\kappa-i\infty}^{\kappa-i\kappa}+ \int_{\kappa+i\kappa}^{\kappa+i\infty}\ex\left(|\zeta(\sigma,X)|^s\right)e^{-\tau s} \frac{\left(e^{\lambda s}-1\right)}{\lambda s} \frac{\left(1-e^{-\lambda s}\right)}{s}ds \ll \frac{e^{-{\kappa}^{1/\sigma-1}}}{\lambda \kappa} \ex\left(|\zeta(\sigma,X)|^{\kappa}\right)e^{-\tau\kappa}.
\end{equation}
Furthermore, if $|t|\leq \kappa$ then $\left|(1-e^{-\lambda s})(e^{\lambda s}-1)\right|\ll \lambda^2|s|^2$. Hence we derive 
$$
\int_{\kappa-i\kappa}^{\kappa+i\kappa} \ex\left(|\zeta(\sigma,X)|^s\right)e^{-\tau s} \frac{\left(e^{\lambda s}-1\right)}{\lambda s} \frac{\left(1-e^{-\lambda s}\right)}{s}ds \ll \lambda\kappa\ex\left(|\zeta(\sigma,X)|^{\kappa}\right)e^{-\tau \kappa}.
$$ 
Therefore, combining this estimate with equations \eqref{approximation1}, \eqref{error12} and \eqref{error2} we deduce that
\begin{equation}\label{approximation2}
\begin{aligned}
\pr(\log |\zeta(\sigma, X)|>\tau) &- \frac{1}{2\pi i}\int_{\kappa-i\kappa}^{\kappa+i\kappa}\ex\left(|\zeta(\sigma,X)|^s\right)e^{-\tau s}\frac{e^{\lambda s}-1}{\lambda s^2} ds \\
&\ll \left(\lambda\kappa+\frac{e^{-\kappa^{1/\sigma-1}}}{\lambda \kappa}\right)\ex\left(|\zeta(\sigma,X)|^{\kappa}\right)e^{-\tau\kappa}.\\
\end{aligned}
\end{equation}
On the other hand, in the region $|t|\leq \kappa$ we have 
$$ 
\log\ex\left(|\zeta(\sigma, X)|^{\kappa+it}\right)= \log\ex\left(|\zeta(\sigma,X)|^{\kappa}\right)+it M'(\kappa)-\frac{t^2}{2}M''(\kappa)+ O\left(M'''(\kappa)|t|^3\right).
$$
Also, note that
$$ \frac{e^{\lambda s}-1}{\lambda s^2}= \frac{1}{\kappa}\left(1-i\frac{t}{\kappa}+ O\left(\lambda \kappa+\frac{t^2}{\kappa^2}\right)\right).$$
Hence, using that $M'(\kappa)=\tau$ we obtain 
\begin{align*}
&\ex\left(|\zeta(\sigma,X)|^s\right)e^{-\tau s}\frac{e^{\lambda s}-1}{\lambda s^2}\\
= &\frac{1}{\kappa}\ex\left(|\zeta(\sigma,X)|^{\kappa}\right)e^{-\tau\kappa}\exp\left(-\frac{t^2}{2}M''(\kappa)\right) 
\left(1-i\frac{t}{\kappa}+O\left(\lambda\kappa+ \frac{t^2}{\kappa^2}+ M'''(\kappa)|t|^3\right)\right)\\
\end{align*}
Therefore, we obtain
\begin{align*}
&\frac{1}{2\pi i}\int_{\kappa-i\kappa}^{\kappa+i\kappa}\ex\left(|\zeta(\sigma,X)|^s\right)e^{-\tau s}\frac{e^{\lambda s}-1}{\lambda s^2} ds\\
=& \frac{1}{\kappa}\ex\left(|\zeta(\sigma,X)|^{\kappa}\right)e^{-\tau\kappa} \frac{1}{2\pi} \int_{-\kappa}^{\kappa}\exp\left(-\frac{t^2}{2}M''(\kappa)\right)
\left(1+ O\left(\lambda\kappa+ \frac{t^2}{\kappa^2}+ M'''(\kappa)|t|^3\right)\right)dt
\end{align*}
since the integral involving $it/{\kappa}$ vanishes. Further, we have 
$$ 
\frac{1}{2\pi} \int_{-\kappa}^{\kappa}\exp\left(-\frac{t^2}{2}M''(\kappa)\right)dt= \frac{1}{\sqrt{2\pi M''(\kappa)}}\left(1+O\left(\exp\left(-\frac12\kappa^2
M''(\kappa)\right)\right)\right),
$$
and 
$$ \int_{-\kappa}^{\kappa}|t|^n\exp\left(-\frac{t^2}{2}M''(\kappa)\right)dt\ll \frac{1}{M''(\kappa)^{(n+1)/2}}.$$
Thus, using Proposition \ref{saddle1} we deduce that
\begin{equation}\label{main}
\begin{aligned}
&\frac{1}{2\pi i}\int_{\kappa-i\kappa}^{\kappa+i\kappa}\ex\left(|\zeta(\sigma,X)|^s\right)e^{-\tau s}\frac{e^{\lambda s}-1}{\lambda s^2} ds\\
=& \frac{\ex\left(|\zeta(\sigma,X)|^{\kappa}\right)e^{-\tau\kappa}}{\kappa\sqrt{2\pi M''(\kappa)}}
\left(1+ O\left(\lambda\kappa+ \kappa^{1-\frac{1}{\sigma}}\log \kappa\right)\right).
\end{aligned}
\end{equation}
Finally, it follows from Proposition \ref{saddle1} that $\kappa\sqrt{M''(\kappa)}\asymp_{\sigma} \kappa^{1/(2\sigma)}(\log\kappa)^{-1/2}.$ Thus, combining the estimates \eqref{approximation2} and \eqref{main} and choosing $\lambda= \kappa^{-3}$ completes the proof.
\end{proof}

\begin{proof}[Proof of Theorem \ref{large thm}]
As before,  $\kappa$ denotes the unique solution to $M'(k)=\tau$. Let $N$ be a positive integer and  $0<\lambda<\min\{1/(2\kappa), 1/N\}$ be a real number to be chosen later. 

Let $A=10$, $\mathcal{E}(T)$, and  $b_4=b_4(\sigma, 10)$ be as in Theorem \ref{complex theorem}. Let $Y=(b_4(\log T)^{\sigma})/2$. Note that, if $T$ is large enough then by Corollary \ref{saddle2}  we have $\kappa\leq Y$. Let $s$ be  a complex number with $\tmop{Re}(s)=\kappa$ and $|\tmop{Im}(s)|\leq Y$.  Then, it follows from Theorem \ref{complex theorem} that
\begin{equation}\label{asymp complex moments}
\frac{1}{T} \int_{[T,2T]\setminus\mathcal{E}(T)} |\zeta(\sigma+it)|^sdt  
= \ex\left(|\zeta(\sigma,X)|^s\right) +O\left(\frac{\ex\left(|\zeta(\sigma,X)|^{\kappa}\right)}{(\log T)^{10}}\right).
\end{equation}
Define
$$
I(\sigma,\tau)= \frac{1}{2\pi i}\int_{\kappa-i\infty}^{\kappa+i\infty}\ex\left(|\zeta(\sigma,X)|^s\right)e^{-\tau s}\left(\frac{e^{\lambda s}-1}{\lambda s}\right)^N\frac{ds}{s}
$$
and 
$$ 
J_T(\sigma,\tau)= \frac{1}{2\pi i}\int_{\kappa-i\infty}^{\kappa+i\infty}\left(\frac{1}{T} \int_{[T,2T]\setminus\mathcal{E}(T)}|\zeta(\sigma+it)|^s dt\right)e^{-\tau s}\left(\frac{e^{\lambda s}-1}{\lambda s}\right)^N\frac{ds}{s}.
$$
 Then, using equation \eqref{indicator} we obtain
\begin{equation}\label{Mellin1}
 \pr(\log |\zeta(\sigma, X)|>\tau)\leq I(\sigma,\tau)\leq \pr(\log |\zeta(\sigma, X)|>\tau-\lambda N),
\end{equation}
and 
\begin{equation}\label{Mellin2}
 \pr_T\Big(\log|\zeta(\sigma+it)|>\tau\Big)+O\big(\delta(T)\big)\leq J_T(\sigma,\tau)\leq \pr_T\Big(\log|\zeta(\sigma+it)|>\tau-\lambda N\Big)+O\big(\delta(T)\big),
\end{equation}
where 
$$
\delta(T)=\exp\left(-c_0(\sigma)\frac{\log T}{\log\log T}\right),
$$
for some positive constant $c_0(\sigma)$, by equation \eqref{bound measure}. 

Further, using that $|e^{\lambda s}-1|\leq 3$ we obtain 
\begin{equation}\label{tail1}
\int_{\kappa-i\infty}^{\kappa-iY}+ \int_{\kappa+iY}^{\kappa+i\infty}\ex\left(|\zeta(\sigma,X)|^s\right)e^{-\tau s}\left(\frac{e^{\lambda s}-1}{\lambda s}\right)^N\frac{ds}{s}\ll \left(\frac{3}{\lambda Y}\right)^N\ex\left(|\zeta(\sigma,X)|^{\kappa}\right)e^{-\tau \kappa}.
\end{equation}
Similarly, using \eqref{asymp complex moments} we get
\begin{equation}\label{tail2}
\begin{aligned}
\int_{\kappa-i\infty}^{\kappa-iY}+ \int_{\kappa+iY}^{\kappa+i\infty}&\left(\frac{1}{T} \int_{[T,2T]\setminus\mathcal{E}(T)}|\zeta(\sigma+it)|^s dt\right)e^{-\tau s}\left(\frac{e^{\lambda s}-1}{\lambda s}\right)^N\frac{ds}{s}\\
&\ll \left(\frac{3}{\lambda Y}\right)^N\ex\left(|\zeta(\sigma,X)|^{\kappa}\right)e^{-\tau \kappa}.
\end{aligned}
\end{equation}
Further, note that $|(e^{\lambda s}-1)/\lambda s|\leq 3$, which is easily seen by looking at the cases $|\lambda s|\leq 1$ and $|\lambda s|>1.$ Therefore,  combining equations  \eqref{asymp complex moments}, \eqref{tail1} and \eqref{tail2}  we obtain
\begin{equation}\label{difference1}
J_T(\sigma, \tau)- I(\sigma, \tau)\ll \ex\left(|\zeta(\sigma,X)|^{\kappa}\right)e^{-\tau \kappa}\left(\frac{3^N Y}{(\log T)^{10}}+ \left(\frac{3}{\lambda Y}\right)^N\right).
\end{equation}
Furthermore, it follows from Proposition \ref{SaddlePoint} and Proposition \ref{saddle1} that
\begin{equation}\label{order}
 \pr(\log |\zeta(\sigma, X)|>\tau)\asymp_{\sigma} \frac{\sqrt{\log\kappa}}{\kappa^{1/(2\sigma)}}\ex\left(|\zeta(\sigma,X)|^{\kappa}\right)e^{-\tau \kappa}
\asymp_{\sigma} \frac{(\log\tau)^{(\sigma+1)/(2\sigma)}}{\tau^{1/(2(1-\sigma))}}\ex\left(|\zeta(\sigma,X)|^{\kappa}\right)e^{-\tau \kappa}.
\end{equation}
Thus, choosing $N=[\log\log T]$ and $\lambda= e^{10}/Y$ we deduce that 
\begin{equation}\label{difference2}
J_T(\sigma, \tau)- I(\sigma, \tau)\ll \frac{1}{(\log T)^{5}} \pr(\log |\zeta(\sigma, X)|>\tau).
\end{equation}
On the other hand, it follows from Corollary \ref{EstimateRandom} that 
\begin{equation}\label{shift} 
\begin{aligned}
\pr(\log |\zeta(\sigma, X)|>\tau\pm \lambda N)&= \pr(\log |\zeta(\sigma, X)|>\tau)\exp(O(\lambda N(\tau\log\tau)^{\frac{\sigma}{1-\sigma}}))\\
&= \pr(\log |\zeta(\sigma, X)|>\tau)\left(1+O\left(\frac{(\tau\log \tau)^{\frac{\sigma}{1-\sigma}}\log\log T}{(\log T)^{\sigma}}\right)\right).\\
\end{aligned}
\end{equation}
Combining this last estimate with \eqref{Mellin1}, \eqref{Mellin2}, and \eqref{difference2} we obtain
\begin{align*}
\pr_T\Big(\log|\zeta(\sigma+it)|>\tau\Big)&
\leq J_T(\sigma,\tau)+O\big(\delta(T)\big) \\
&\leq I(\sigma,\tau)+ O\left(\frac{\pr(\log |\zeta(\sigma, X)|>\tau)}{(\log T)^{5}}+ \delta(T)\right)\\
&\leq \pr(\log |\zeta(\sigma, X)|>\tau)\left(1+O\left(\frac{(\tau\log \tau)^{\frac{\sigma}{1-\sigma}}\log\log T}{(\log T)^{\sigma}}\right)\right)+ O\big(\delta(T)\big),
\end{align*}
and 
\begin{align*}
\pr_T\Big(\log|\zeta(\sigma+it)|>\tau\Big)&
\geq J_T(\sigma,\tau+\lambda N)+O\big(\delta(T)\big) \\
&\geq I(\sigma,\tau+\lambda N)+ O\left(\frac{\pr(\log |\zeta(\sigma, X)|>\tau)}{(\log T)^{5}}+ \delta(T)\right)\\
&\geq \pr(\log |\zeta(\sigma, X)|>\tau)\left(1+O\left(\frac{(\tau\log \tau)^{\frac{\sigma}{1-\sigma}}\log\log T}{(\log T)^{\sigma}}\right)\right)+ O(\delta(T)).\\
\end{align*}
The result follows from these estimates together with the fact that  $\pr(\log |\zeta(\sigma, X)|>\tau)\gg (\delta(T))^{1/2}$ in our range of $\tau$, by Corollary \ref{Decay}.

\end{proof}

\section{Distribution of $a$-points: Proof of Theorem 1.3}

\subsection{Preliminaries}

To shorten our notation we let $\log_2T=\log\log T$. Let $S(T)$ be the set of points $T \leq t \leq 2T$ such that
 $$\max\left\{\big|\log|\zeta(\sigma+it)|\big|, \big|\arg\zeta(\sigma+it)\big|\right\}<\log_2T \text{ and } \Big|\log|\zeta(\sigma+it)|-\log|a|\Big|>\delta,$$
where $\delta=1/(\log T)^{\sigma}$.
Similarly let $F$ be the event, 
$$ \max\left\{\big|\log|\zeta(\sigma, X)|\big|, \big|\arg\zeta(\sigma,X)\big|\right\}<\log_2T \text{ and } \Big|\log|\zeta(\sigma,X)|-\log|a|\Big|>\delta.
$$
Then we have the following lemma.

\begin{lemma} \label{initial}
Let $\tfrac12<\sigma<1$ be fixed.
We have,
$$
\frac{1}{T}
\int_{T}^{2T} \log |\zeta( \sigma + it) - a| dt = \frac{1}{T} \int_{t \in S(T)}
\log |\zeta(\sigma + it) - a| dt + O \bigg ( \frac{(\log_2 T)^{2}}{(\log T)^{\sigma}} \bigg ).
$$
and
$$
\mathbb E(\log |\zeta(\sigma, X) - a|) = \mathbb{E} ( \mathbf{1}_{F} \cdot \log |\zeta(\sigma,X) - a|)
+ O \bigg ( \frac{(\log_2 T)^2}{(\log T)^{ \sigma}} \bigg ). 
$$
\end{lemma}
\begin{proof}
Note
\begin{align*}
& \bigg |
\frac{1}{T}
\int_{t \notin S(T)} \log |\zeta(\sigma + it) - a| dt \bigg | \leq \\ & \leq \bigg ( \frac{1}{T} \cdot \text{meas} \{ T \leq t \leq 2T: t \notin S(T) \}  \bigg )^{1 - 1/2k} \cdot \bigg ( \frac{1}{T} \int_{T}^{2T} | \log |\zeta(\sigma + it) - a||^{2k} \, dt \bigg )^{1/2k}.
\end{align*}
According to Proposition \ref{apoint prop}, 
$$ \left(\frac{1}{T}\int_T^{2T}(\log|\zeta(\sigma+it)-a|)^{2k}dt\right)^{1/2k} \ll k^2$$ 
while by Theorem \ref{discrepancy thm} we have 
$$
\text{meas} \{ T \leq t \leq 2T: t \notin S(T) \} \ll \mathbb{P} \big ( | \log |\zeta(\sigma, X)| - \log |a| | < \delta\big ) + O((\log T)^{-\sigma}).
$$
The probability distribution $\mathbb{P} ( \log \zeta(\sigma, X) \in \cdot)$ is absolutely continuous, and therefore
the above expression is $\ll \delta + (\log T)^{-\sigma} \ll(\log T)^{-\sigma}  $ . 
Choosing $k = \log_2 T$ leads to the desired estimate
$$
\bigg |
\int_{t \notin S(T)} \log |\zeta(\sigma + it) - a| dt \bigg | \ll \frac{(\log_2 T)^{2}}{(\log T)^{\sigma}} 
$$
and hence the claim. The proof of the second statement is similar.
\end{proof}

We let $S_1(T)$ be the set of points $t\in S(T)$ such that $\log|\zeta(\sigma+it)|>\log|a|+\delta$, and 
$S_2(T)=S(T)\setminus S_1(T)$. Similarly, $F_1$ is the sub-event of $F$ where $\log|\zeta(\sigma,X)|>\log|a|+\delta$ and $F_2=F\setminus F_1.$ Moreover, we define
\begin{align*}
\Phi_1(u,v)&= \frac{1}{T} \text{meas}\{t\in S_1(T) : \log|\zeta(\sigma+it)|\leq u \text{ and } \arg\zeta(\sigma+it)\leq v\}\\
\tilde{\Phi}_1(u,v)&= \mathbb{P}\big(F_1 \text{ and } \log |\zeta(\sigma, X)|\leq u \text{ and } \arg\zeta(\sigma, X)\leq v\big).
\end{align*} 
Also, let
\begin{align*}
\Psi(u)&= \frac{1}{T} \text{meas}\{t\in S_1(T) : \log|\zeta(\sigma+it)|\leq u\}\\
\tilde{\Psi}(u)&= \mathbb{P}\big(F_1 \text{ and } \log|\zeta(\sigma, X)|\leq u\big).
\end{align*}

Let $g(u,v):= \log(e^{u+iv}-a)$ and $h(u,v):= \tmop{Re} (g(u,v))$. Note that $h$ is twice differentiable in the region of $\mathbb{R}^2$ where $\big| u-\log|a|\big|>\delta$. 

We are now going to show that 
$$
\int_{t \in S(T)} \log |\zeta(\sigma + it) - a| dt \text{ and }
\mathbb{E}[\mathbf{1}_{F} \cdot \log |\zeta(\sigma, X) - a| ]
$$
match up to a small error term. For this we will need to integrate by parts. We establish the three necessary lemmas below. 
\begin{lemma}\label{IntParts} We have
\begin{align*}
 &\frac{1}{T}\int_{t\in S_1(T) }\log|\zeta(\sigma+it)-a|dt\\
=& \int_{-\log_2T}^{\log_2T}\int_{\log|a|+\delta}^{\log_2 T} \Phi_1(u,v)
\frac{\partial^2 h(u,v)}{\partial u\partial v}dudv- \frac{\textup{meas}(S_1(T))}{T} h(\log_2T,\log_2T)\\
& + \frac{1}{T}\int_{t\in S_1(T) }\Big(h\big(\log_2T, \arg\zeta(\sigma+it)\big)+h\big(\log|\zeta(\sigma+it)|,\log_2T\big)\Big)dt,\\
\end{align*} 
and 
\begin{align*}
 &\ex\big(\mathbf{1}_{F_1}\cdot\log|\zeta(\sigma, X)-a|\big)\\
=& \int_{-\log_2T}^{\log_2T}\int_{\log|a|+\delta}^{\log_2 T} \tilde{\Phi}_1(u,v)
\frac{\partial^2 h(u,v)}{\partial u\partial v}dudv- \pr(F_1)h(\log_2T,\log_2T)\\
& + \ex\Big(\mathbf{1}_{F_1}\cdot h\big(\log_2T, \arg\zeta(\sigma,X)\big)\Big) + \ex\Big(\mathbf{1}_{F_1}\cdot h\big(\log|\zeta(\sigma,X)|,\log_2T\big)\Big).\\
\end{align*} 

\end{lemma}
\begin{proof}
We only prove the first identity since the second can be obtained along similar lines. We have
\begin{align*}
&\int_{-\log_2T}^{\log_2T}\int_{\log|a|+\delta}^{\log_2 T} \Phi_1(u,v)
\frac{\partial^2 h(u,v)}{\partial u\partial v}dudv\\
=& \frac{1}{T} \int_{-\log_2T}^{\log_2T}\int_{\log|a|+\delta}^{\log_2 T} \int_{\substack{t\in S_1(T)\\ \log|\zeta(\sigma+it)|\leq u\\ \arg\zeta(\sigma+it)\leq v} }
\frac{\partial^2 h(u,v)}{\partial u\partial v}dtdudv\\
=& \frac{1}{T} \int_{t\in S_1(T) }
\int_{\arg\zeta(\sigma+it)}^{\log_2T}\int_{\log|\zeta(\sigma+it)|}^{\log_2 T} 
\frac{\partial^2 h(u,v)}{\partial u\partial v}dudvdt\\
=&\frac{1}{T} \int_{t\in S_1(T) }
\int_{\arg\zeta(\sigma+it)}^{\log_2T} 
\frac{\partial }{\partial v}h\big(\log_2T,v\big)-
\frac{\partial}{\partial v}h\big(\log|\zeta(\sigma+it)|,v\big)dvdt\\
=& \frac{1}{T}\int_{t\in S_1(T) }\log|\zeta(\sigma+it)-a|dt + \frac{1}{T} \int_{t\in S_1(T) } h\big(\log_2T,\log _2T\big)dt \\
&-\frac{1}{T}\int_{t\in S_1(T) }\Big(h\big(\log_2T, \arg\zeta(\sigma+it)\big)+h\big(\log|\zeta(\sigma+it)|,\log_2T\big)\Big)dt.
\end{align*}
\end{proof}
\begin{lemma}\label{ParInt1} 
Let $\tfrac12<\sigma<1$ be fixed.
We have 
$$\frac{1}{T}\int_{t\in S_1(T) }h\big(\log|\zeta(\sigma+it)|,\log_2T\big)dt=\ex\Big(\mathbf{1}_{F_1} \cdot h\big(\log|\zeta(\sigma, X)|, \log_2T\big)\Big)+O\left(\frac{\log_2T}{(\log T)^{\sigma}}\right),$$
and 
$$ \frac{1}{T}\int_{t\in S_1(T) }h\big(\log_2T, \arg\zeta(\sigma+it)\big)dt = \ex\Big(\mathbf{1}_{F_1} \cdot h\big(\log_2T, \arg\zeta(\sigma, X)\big)\Big) +O\left(\frac{\log_2T}{(\log T)^{\sigma}}\right).$$
\end{lemma}
\begin{proof}
 We only prove the first estimate since the second is similar. We have  
\begin{align*}
 \frac{1}{T}\int_{t\in S_1(T) }h\big(\log|\zeta(\sigma+it)|,\log_2T\big)dt
= \int_{\log|a|+\delta}^{\log_2T} h(u,\log_2T)d\Psi(u).
\end{align*}
Integrating by parts, the right-hand side equals
\begin{align*}
 &\Big[\Psi(u)h(u,\log_2T)\Big]_{\log|a|+\delta}^{\log_2T}- \int_{\log|a|+\delta}^{\log_2T}h'(u,\log_2T)\Psi(u)du\\
& \qquad \qquad \qquad \qquad = \Big[\tilde{\Psi}(u)h(u,\log_2T)\Big]_{\log|a|+\delta}^{\log_2T}- \int_{\log|a|+\delta}^{\log_2T}h'(u,\log_2T)\tilde{\Psi}(u)du + E_5\\
& \qquad \qquad \qquad \qquad = \ex\Big(\mathbf{1}_{F_1} \cdot h\big(\log|\zeta(\sigma, X)|, \log_2T\big)\Big)+E_5
\end{align*}
where 
\begin{equation}\label{error13}
 E_5\ll \frac{1}{(\log T)^{\sigma}}\left(\log_2T+ \int_{\log|a|+\delta}^{\log_2T}|h'(u,\log_2T)|du\right),
\end{equation}
which follows from the discrepancy estimate $ \Psi(u)-\tilde{\Psi}(u)\ll (\log T)^{-\sigma}$, along with the bounds
$h(\log_2T,\log_2T)\ll \log_2T$ and $h(\log|a|+\delta, \log_2T)\ll \log(1/\delta)\ll \log_2T.$
Now, we have 
$$ |h'(u,\log_2T)|= |\tmop{Re} (g'(u, \log_2T))|\leq |g'(u,\log_2T)|\leq \frac{e^u}{|e^{u}-|a||}.$$
Further, by making the change of variable $x= u-\log|a|$, we get
$$
\int_{\log|a|+\delta}^{\log_2T}|h'(u,\log_2T)|du \ll \int_{\delta}^{2\log_2T}\frac{e^x}{e^x-1}dx
\ll \int_{\delta}^{1}\frac{dx}{x} + \log_2T \ll \log_2 T.
$$
Inserting this estimate in \eqref{error13} completes the proof.
\end{proof}

\bigskip

\begin{lemma}\label{ParInt2} Let $\tfrac12<\sigma<1$ be fixed. We have
\begin{align*}
\int_{-\log_2T}^{\log_2T}\int_{\log|a|+\delta}^{\log_2 T} \Phi_1(u,v)
\frac{\partial^2 h(u,v)}{\partial u\partial v}dudv=& 
\int_{-\log_2T}^{\log_2T}\int_{\log|a|+\delta}^{\log_2 T} \tilde{\Phi}_1(u,v)
\frac{\partial^2 h(u,v)}{\partial u\partial v}dudv \\
&+O \left(\frac{(\log_2 T)^2}{(\log T)^{\sigma}}\right).
\end{align*}

\end{lemma}

\begin{proof}
By the discrepancy estimate 
$ \Phi_1(u,v)-\tilde{\Phi}_1(u,v)\ll (\log T)^{-\sigma}$, we obtain that 
\begin{align*}
\int_{-\log_2T}^{\log_2T}\int_{\log|a|+\delta}^{\log_2 T} \Phi_1(u,v)
\frac{\partial^2 h(u,v)}{\partial u\partial v}dudv
=&\int_{-\log_2T}^{\log_2T}\int_{\log|a|+\delta}^{\log_2 T} \tilde{\Phi}_1(u,v)
\frac{\partial^2 h(u,v)}{\partial u\partial v}dudv\\
&+ O\left( \frac{1}{(\log T)^{\sigma}} \int_{\log|a|+\delta}^{\log_2 T} \int_{-\log_2T}^{\log_2T}
\left|\frac{\partial^2 h(u,v)}{\partial u\partial v}\right|dvdu\right).
\end{align*}
Note that
$$\left|\frac{\partial^2 h(u,v)}{\partial u\partial v}\right|= \left|\tmop{Re}\frac{\partial^2 g(u,v)}{\partial u\partial v} \right|\leq \left|\frac{\partial^2 g(u,v)}{\partial u\partial v}\right|\ll \frac{e^u}{|e^{u+iv}-a|^2},$$
and $|e^{u+iv}-a|^2= e^{2u}+|a|^2-2\tmop{Re}(a e^{u-iv})= (e^u-|a|)^2+2|a|e^u\big(1-\cos(v-\arg a)\big).$
We split the range of integration over $v$ into intervals $[-\pi+ 2\pi k+\arg a, \pi+2\pi k +\arg a]$ with 
$|k|\leq (\log_2 T)/\pi$. Since the integrand is non-negative, we deduce that 
\begin{align*}
 \int_{-\log_2T}^{\log_2T}
\left|\frac{\partial^2 h(u,v)}{\partial u\partial v}\right|dv 
&\leq e^{u}\sum_{|k|\leq \log_2T} \int_{-\pi+2\pi k+\arg a}^{\pi+2\pi k +\arg a} \frac{1}{(e^u-|a|)^2+2|a|e^u\big(1-\cos(v-\arg a)\big)}dv \\
&\ll e^u \log_2 T \int_0^{\pi} \frac{1}{(e^u-|a|)^2+2|a|e^u(1-\cos v)}dv,
\end{align*}
by a simple change of variable and since the integrand is an even function of $v$. Furthermore, using that 
$1-\cos v\geq v^2/10$ for $0\leq v\leq \pi$ we obtain that 
$$ \int_0^{\pi} \frac{1}{(e^u-|a|)^2+2|a|e^u(1-\cos v)}dv \leq \int_{0}^{\pi} 
\frac{1}{(e^u-|a|)^2+|a|e^u v^2/5}dv.$$
Now, by making the change of variable 
$$ y= \frac{\sqrt{|a|}e^{u/2}}{\sqrt{5}(e^u-|a|)}v,$$
we derive
$$ \int_{0}^{\pi} 
\frac{1}{(e^u-|a|)^2+|a|e^u v^2/5}dv\ll \frac{e^{-u/2}}{e^u-|a|}\int_0^{\infty}\frac{1}{1+y^2}dy \ll \frac{e^{-u/2}}{e^u-|a|}.
$$
Combining these estimates we deduce that
$$
\int_{\log|a|+\delta}^{\log_2 T} \int_{-\log_2T}^{\log_2T}
\left|\frac{\partial^2 h(u,v)}{\partial u\partial v}\right|dvdu\ll
 \log_2 T \int_{\log|a|+\delta}^{\log_2 T} \frac{e^{u/2}}{e^u-|a|}du\ll \log_2 T \int_{\delta}^{2\log_2 T} \frac{e^{x/2}}{e^x-1}dx
$$
by making the change of variable $x= u-\log|a|$ and since the integrand is positive. The lemma follows upon noting that
$$  
\int_{\delta}^{2\log_2 T} \frac{e^{x/2}}{e^x-1}dx \ll \int_{\delta}^1\frac{1}{x}dx + \int_1^{\log_2T} e^{-x/2}dx 
\ll \log_2T.
$$
\end{proof}
\subsection{Proofs of Theorems \ref{aux} and  \ref{apoint thm}}
\begin{proof}[Proof of Theorem \ref{aux}]
In view of Lemma \ref{initial} we only need to prove that
$$\frac{1}{T}\int_{t\in S(T) }\log|\zeta(\sigma+it)-a|dt= \ex\big(\mathbf{1}_{F}\cdot\log|\zeta(\sigma, X)-a|\big) + O \left(\frac{(\log_2 T)^2}{(\log T)^{\sigma}}\right).$$
Recall that $S(T)=S_1(T) \cup S_2(T)$ and $F=F_1 \cup F_2$. Combining the discrepancy estimate $$ \frac{\text{meas}(S_1(T))}{T}-\mathbb{P}(F_1)\ll (\log T)^{-\sigma}$$ with Lemmas \ref{IntParts}, \ref{ParInt1} and \ref{ParInt2}, we obtain
$$\frac{1}{T}\int_{t\in S_1(T) }\log|\zeta(\sigma+it)-a|dt= \ex\big(\mathbf{1}_{F_1}\cdot\log|\zeta(\sigma, X)-a|\big) + O \left(\frac{(\log_2 T)^2}{(\log T)^{\sigma}}\right).$$

Finally, using a similar approach one obtains
$$\frac{1}{T}\int_{t\in S_2(T) }\log|\zeta(\sigma+it)-a|dt=\ex\big(\mathbf{1}_{F_2}\cdot\log|\zeta(\sigma, X)-a|\big) + O \left(\frac{(\log_2 T)^2}{(\log T)^{\sigma}}\right),$$
where instead of Lemma \ref{IntParts} we use
\begin{align*}
 &\frac{1}{T}\int_{t\in S_2(T) }\log|\zeta(\sigma+it)-a|dt\\
=& \int_{-\log_2T}^{\log_2T}\int_{-\log_2 T}^{\log|a|-\delta} \Phi_2(u,v)
\frac{\partial^2 h(u,v)}{\partial u\partial v}dudv- \frac{\textup{meas}(S_2(T))}{T} h(-\log_2T,-\log_2T)\\
& + \frac{1}{T}\int_{t\in S_2(T) }\Big(h\big(-\log_2T, \arg\zeta(\sigma+it)\big)+h\big(\log|\zeta(\sigma+it)|,-\log_2T\big)\Big)dt,\\
\end{align*} 
with 
$$ \Phi_2(u,v)= \frac{1}{T} \text{meas}\{t\in S_2(T) : \log|\zeta(\sigma+it)|\geq u \text{ and } \arg\zeta(\sigma+it)\geq v\}.$$
\end{proof}

For the proof of Theorem \ref{apoint thm} we need an auxiliary lemma.
\begin{lemma}\label{Diff}
Let $a \neq 0$.
The function
$$
f_{a}(\sigma) := \mathbb{E}[\log |\zeta(\sigma,X) - a|]
$$
is twice differentiable in $\sigma$ for
$\tfrac 12 < \sigma < 1$. 
\end{lemma}

\begin{proof}
See Theorem 14 of \cite{BorchseniusJessen}.
\end{proof}

\begin{proof}[Proof of Theorem \ref{apoint thm}]

Let $\frac12<\sigma<1$ and $\rho_a = \beta_a + i \gamma_a$ denote an $a$-point of $\zeta(s)$. We know that there is $\sigma_0=\sigma_0(a)$ such that $\beta_a<\sigma_0$ for all $a$-points $\rho_a$. By Littlewood's lemma
(see equation (9.9.1) of Titchmarsh \cite{Titchmarsh}), we have 
\begin{equation}\label{Littlewood}
\begin{aligned}
\int_{\sigma}^{\sigma_0}\bigg(\sum_{\substack{ \beta_a >u \\ T \leq \gamma_a \leq 2T}} 1\bigg) du
=& \frac{1}{2\pi} \int_{T}^{2T} \log |\zeta(\sigma + it) - a| dt - \frac{1}{2\pi} \int_{T}^{2T} \log |\zeta(\sigma_0 + it) - a| dt\\
&+\frac{1}{2\pi}\int_{\sigma}^{\sigma_0}
\Big(\arg\big(\zeta(\alpha+2iT)-a\big)-\arg\big(\zeta(\alpha+iT)-a\big)\Big)d\alpha. 
\end{aligned}
\end{equation}
Furthermore, a standard application of the argument principle shows that (see for example equation (8.4) of Tsang's Thesis \cite{Tsang}) 
$$ \int_{\sigma}^{\sigma_0}
\Big(\arg\big(\zeta(\alpha+2iT)-a\big)-\arg\big(\zeta(\alpha+iT)-a\big)\Big)d\alpha\ll_a \log T.
$$
Let $0<h<\min(\sigma-\frac12, 1-\sigma)$. Inserting this last estimate in equation \eqref{Littlewood} and using Theorem \ref{aux} we obtain
$$
\int_{\sigma}^{\sigma+h}\bigg(\sum_{\substack{ \beta_a >u \\ T \leq \gamma_a \leq 2T}} 1\bigg)  du=
\frac{T}{2\pi} \cdot \bigg(\mathbb{E} [ \log |\zeta(\sigma, X) - a|] - \mathbb{E} [ \log |\zeta(\sigma+h, X) - a|]\bigg)+ O \bigg (\frac{T (\log_2 T)^{2}}{(\log T)^{\sigma}} 
\bigg).
$$
Recall that $f_a(\sigma) = \mathbb{E} [\log |\zeta(\sigma, X) - a|]$ is twice differentiable in $\sigma$ by Lemma \ref{Diff}. Hence, we derive
\begin{align*}
\frac{1}{h} \int_{\sigma}^{\sigma + h} \bigg(\sum_{\substack{ \beta_a >u \\ T \leq \gamma_a \leq 2T}} 1\bigg) du & = \frac{T}{2\pi} \cdot \bigg ( \frac{f(\sigma) - f(\sigma + h)}{h} \bigg ) + O \bigg ( \frac{T (\log_2 T)^{2}}{
(\log T)^{\sigma}} \cdot \frac{1}{h} \bigg ) \\
& = - \frac{T}{2\pi} \cdot f'(\sigma) + O \bigg ( h T + \frac{T (\log_2 T)^{2}}{(\log T)^{\sigma}} \cdot \frac{1}{h} \bigg ).
\end{align*}
Therefore,
$$
\sum_{\substack{\beta_a \geq \sigma + h \\ T \leq \gamma_a \leq 2T}} 1 \leq
- \frac{T}{2\pi} \cdot f'(\sigma) +  O \bigg ( h T + \frac{T (\log_2 T)^{2}}{(\log T)^{\sigma}} \cdot \frac{1}{h} \bigg )
\leq \sum_{\substack{\beta_a \geq \sigma \\ T \leq \gamma_a \leq 2T}} 1.
$$
We substitute $\sigma - h$  for $\sigma$ and use $f'(\sigma - h) = f'(\sigma) + O(h)$ to conclude that also
$$
\sum_{\substack{\beta_a \geq \sigma \\ T \leq \gamma_{a} \leq 2T}} 1
\leq - \frac{ T}{2\pi} \cdot f'(\sigma) + O \bigg ( h T + \frac{T (\log_2 T)^{2}}{(\log T)^{\sigma}} \cdot \frac{1}{h} \bigg ).
$$
We pick $h = (\log_2 T) \cdot (\log T)^{-\sigma/2}$ to conclude that
$$
\sum_{\substack{\beta_a \geq \sigma \\ T \leq \gamma_a \leq 2T}} 1 = 
- \frac{T}{2\pi} \cdot f'(\sigma) + O \bigg ( \frac{T \log_2 T}{(\log T)^{\sigma/2}} \bigg ).
$$
From this the claim follows. 
\end{proof}

\section{Appendix: Lower bounds for the discrepancy}

According to \cite{Ivic},
$$
\int_{T}^{2T} |\zeta(\sigma + it)|^2 dt = \zeta(2\sigma) T
+ (2\pi)^{2\sigma - 1} \cdot \frac{\zeta(2 - 2\sigma)}{2 - 2\sigma}
\cdot ( 2^{2 - 2\sigma} - 1) T^{2 - 2\sigma} + O(T^{1 - \sigma}). 
$$
We notice that
$$
\mathbb{E}(|\zeta(\sigma,X)|^{2}) = \zeta(2\sigma).
$$
Therefore, if $D_{\sigma}(T) = O(T^{1 - 2\sigma - \delta})$ for some $\delta > 0$, then
by integration by parts
$$
\int_{T}^{2T} |\zeta(\sigma + it)|^2 dt = \zeta(2\sigma) T + O(T^{2 - 2\sigma - \delta})
$$
which contradicts the previous equation. Therefore $D_{\sigma}(T) = \Omega(T^{1 - 2\sigma - \varepsilon})$.
We notice that the term $T^{2 - 2\sigma}$ arises from the $\chi$ factors in the approximate
functional equation. Therefore the observed discrepancy $D_{\sigma}(T) = \Omega(T^{1 - 2\sigma - \varepsilon})$ ultimately arises because the probabilistic model $\zeta(\sigma, X)$ does not take into
account the $\chi$ factors in the approximate functional equation (or equivalently because
independence is ruined for the harmonics $n^{it}$ and $m^{it}$ with $n,m$ close to $T$).

As to the second assertion, if we have that $D_{\sigma}(T) = O(T^{1 - 2\sigma + \varepsilon})$, then
again an integration by parts shows that
$$
\int_{T}^{2T} \log |\zeta(\sigma + it)| dt = T \cdot \mathbb{E}[\log |\zeta(\sigma, X)|] + O(T^{2 - 2\sigma
+ \varepsilon}).
$$
Since $\log |\zeta(\sigma,X)|$ is symmetric we have $\mathbb{E}[\log |\zeta(\sigma, X)|] = 0$. 
By Littlewood's lemma we conclude that
$$
\sum_{\substack{\beta > \sigma \\ T  \leq \gamma \leq 2T}} (\beta - \sigma) = O(T^{2-2\sigma + \varepsilon}).
$$
From this it follows that the number of zeros of $\zeta(s)$ in the region $\beta > \sigma + \varepsilon$
is $\ll T^{2 - 2\sigma + \varepsilon}$ as desired. 

\section*{Acknowledgments}

We would like to thank Yoonbok Lee for discussions regarding
the generalization of our result to Epstein zeta-functions and linear
combinations of $L$-functions. 

\bibliographystyle{amsplain}
\bibliography{Paper}

\end{document}